\documentclass[final,12pt,times]{elsarticle}

\usepackage{booktabs} 
\usepackage{epsfig}
\usepackage{amsfonts}
\usepackage{amssymb}
\usepackage{amsmath}
\usepackage{amssymb}
\usepackage{amsthm}
\usepackage{graphics}
\usepackage{graphicx}
\usepackage{subcaption}
\usepackage{mathrsfs}
\usepackage{indentfirst}
\usepackage{stmaryrd}
\usepackage{latexsym}
\usepackage{xypic}
\usepackage{tikz}
\usepackage{hyperref}
\usepackage{listings}
\usepackage{algorithm}
\usepackage{algorithmic}
\usepackage{tikz-cd}
\usepackage{amsmath,amssymb}
\usepackage{mathrsfs}
\usepackage{CJK}
\biboptions{sort&compress}
\theoremstyle{plain}
\theoremstyle{definition}
\newtheorem{definition}{Definition}[section]
\newtheorem{lemma}[definition]{Lemma}
\newtheorem{openproblem}[definition]{Open problem}
\newtheorem{theorem}[definition]{Theorem}
\newtheorem{proposition}[definition]{Proposition}
\newtheorem{example}[definition]{Example}

\newtheorem{remark}[definition]{Remark}
\newtheorem{corollary}[definition]{Corollary}
\theoremstyle{break}

\journal{Annals of Pure and Applied Logic}

\newcommand{\FARL}{\mathbb{FARL}}
\newcommand{\FAIRL}{\mathbb{FAIRL}}
\newcommand{\MRL}{\mathbb{MRL}}
\newcommand{\FARDL}{\mathbb{FARDL}}

\newcommand{\CFAD}{\mathbb{FA\text{-}cDRDL}}
\newcommand{\CFADp}{\mathbb{FA\text{-}cDRDL}^{\prime}}
\newcommand{\FLew}{\mathbf{FL}_{ew}}
\newcommand{\FLFGC}{\mathbf{FL}^{\mathrm{FGC}}_{ew}}
\newcommand{\Kfun}{\mathbf K}
\newcommand{\Cfun}{\mathbf C}
\newcommand{\Id}{\operatorname{Id}}
\makeatletter
\newcommand{\FARLtag}[1]{%
  \phantomsection%
  \def\@currentlabel{\textnormal{(FARL#1)}}%
  \label{FARL#1}%
  \textnormal{(FARL#1)}%
}
\makeatother
\makeatletter
\newcommand{\fKtag}[1]{%
  \phantomsection%
  \protected@edef\@currentlabel{%
    \noexpand\ensuremath{(f_K#1)}%
  }%
  \label{fK#1}%
  \ensuremath{(f_K#1)}%
}
\makeatother
\makeatletter
\newcommand{\fctag}[1]{%
  \phantomsection%
  \protected@edef\@currentlabel{%
    \noexpand\ensuremath{(f_c#1)}%
  }%
  \label{fc#1}%
  \ensuremath{(f_c#1)}%
}
\makeatother

\begin{document}

\begin{frontmatter}


\title{\textbf{Frobenius–Galois expansions of substructural logics:
Algebraization, Kalman equivalence and positive-cone
semantics}}

\author{Juntao Wang$^{a,*}$, Jieqiong Shi$^{b}$, Mei Wang$^{c}$}
\cortext[cor1]{Corresponding author. \\
Email addresses: wjt@xsyu.edu.cn (J. T. Wang), jieqiong6912@126.com (J. Q. Shi),\\ wangmeimath@163.com (M. Wang).}

\address[A] {School of Science, Xi'an Petroleum University, Xi'an 710065, Shaanxi, P.R. China}
\address[B]{School of Mathematics and Data Science, Shaanxi University of Science and Technology, Xi’an 710021, China}
\address[C]{School of Science, Xi Hang University, Xi'an 710077, Shaanxi, P.R. China}

\begin{abstract}
The main aim of this paper is to introduce a Frobenius--Galois expansion of the substructural logic 
$\mathbf{FL}_{ew}$ and develop its algebraic and categorical semantics. 
The resulting logic, denoted by $\mathbf{FL}^{\mathbf{FGC}}_{ew}$, is 
obtained by adjoining a pair of unary connectives forming a Galois 
connection and satisfying suitable Frobenius-type compatibility 
conditions. Firstly, we prove that $\mathbf{FL}^{\mathbf{FGC}}_{ew}$ is 
algebraizable in the sense of Blok and Pigozzi and identify its 
equivalent algebraic semantics with the variety of Frobenius-adjoint 
residuated lattices, establishing the conservativity over 
$\mathbf{FL}_{ew}$ and relating its finite model property to residual 
finiteness of finitely generated free algebras. Secondly, for the distributive setting, we lift the Kalman construction to the 
Frobenius-adjoint algebraic framework. More precisely, we establish a categorical 
equivalence between Frobenius-adjoint residuated distributive lattices 
and Frobenius-adjoint $c$-differential residuated distributive lattices 
with the condition $\mathbf{CK}$. This equivalence yields a 
positive-cone representation of the former structures and, in turn, a 
logical counterpart of the categorical correspondence. Finally, for the 
distributive extension $\mathbf{FL}^{\mathbf{FGC},d}_{ew}$, we prove 
that derivability, validity over Frobenius-adjoint residuated 
distributive lattices, and validity over the corresponding positive 
cones determine the same consequence relation, and further show that 
this correspondence preserves equational and quasi-equational 
consequence, conservativity, and finite countermodels. These results 
provide a unified algebraic, categorical, and logical framework for 
Frobenius--Galois expansions of substructural logics.
\end{abstract}


\end{frontmatter}

\section{Introduction}
\label{intro}

The reader is expected to be familiar with the fundamental aspects of the Kalman construction in the setting of residuated lattices, as presented in \cite{Castiglionib,Castiglionic,Sagastume,Sagastume1,Tsinakis}.
  Substructural logics form a family of nonclassical logics that can be roughly defined as those logic systems such that, when presented by means of a Gentzen-style calculus, lack some of the structural rules (exchange, contraction and weakening), which occur in standard sequent systems of classical and intuitionistic logics \cite{Galatos,Metcalfe}. As such they encompass a variety of systems independently developed since mid 20th century, including relevant logics or many-valued logics like monoidal logic (not satisfying contraction), linear logic (which, besides the former two, does not satisfy exchange either).  The  Gentzen system $\mathbf{FL}$ of full Lambek calculus, the weakest logic considered in substructural logics, is obtained from the Gentzen system $\mathbf{LJ}$ for intuitionistic propositional logic by removing three structural rules. As a very important extension of $\mathbf{FL}$, the sequent calculus $\mathbf{FL_{ew}}$ is obtained from the intuitionistic logic by only deleting the contraction rule. Researchers always call \emph{substructural logics over $\mathbf{FL_{ew}}$} or \emph{logics without contraction rule} \cite{Ono}, although the contraction rule holds in some of them. The class of logics without the contraction rule contains intermediate logics, {\L}ukasiewicz's many-valued logics \cite{Chang} and fuzzy logics in the sense of H\'{a}jek \cite{Hajek}. Our main technical tool is algebraic in nature and is based on close connections between logics over $\mathbf{FL_{ew}}$ and the variety of residuated lattices \cite{Blount,Metcalfe,Cintula1,Galatos,Ono}.
  
  Galois connections provide one of the most elementary order-theoretic
forms of adjunction and have been used to organize pairs of
approximation, modal and quantifier-like operators on ordered
algebraic structures\cite{Ore1944,JarvinenKondoKortelainen2008,DzikJarvinenKondo2010}.
For a pair of the two maps $g,f$ on a partially ordered algebra, 
\[
    g(x)\leq y
    \Longleftrightarrow
    x\leq f(y)
\]
already entails strong order-theoretic consequences: the left adjoint
$g$ preserves joins, whereas the right adjoint $f$ preserves meets.
Nevertheless, adjointness alone says essentially nothing about how the
two unary operators interact with the additional algebraic operations.
This becomes particularly relevant for residuated lattices, whose
multiplicative and residual operations form an adjoint pair,
\[
    x\odot y\leq z
    \Longleftrightarrow
    y\leq x\to z
\]
Thus, once a second adjunction $g\dashv f$ is placed on a residuated
lattice, a natural structural question is not merely whether the unary
operators are adjoint, but whether this new adjunction is compatible
with the intrinsic residuated adjunction. A classical paradigm for such compatibility is provided by
\emph{Frobenius reciprocity}. In categorical logic, existential
quantification is represented by a left adjoint to substitution, and
its interaction with conjunction is expressed by a Frobenius law of
the form
\[
   \exists_x(\varphi\sqcap\psi)
   =
   (\exists_x\varphi)\sqcap\psi
\]
whenever $\psi$ is independent of $x$, this adjoint interpretation goes
back to Lawvere~\cite{Lawvere1969}. More generally, Frobenius
reciprocity can be viewed as an additional compatibility condition on
an adjoint connection rather than as a consequence of adjointness
itself, its close relationship with Galois connections, residuation,
and lattice-theoretic structures has been emphasized recently in \cite{GoswamiJanelidzeManuell2025}.
This observation suggests a natural refinement in the setting of
residuated lattices: the unary Galois connection should be required to
interact coherently with the residual and lattice operations, rather
than being added to the reduct as an independent order-theoretic
structure. This motivates the introduction of \emph{Frobenius-adjoint residuated lattices}, where the interaction between the unary Galois adjunction and the residuated adjunction is captured by suitable Frobenius-type compatibility laws.

Accordingly, the first aim of this paper is to introduce and systematically investigate Frobenius-adjoint residuated lattices from both algebraic and logical perspectives. In Section~\ref{section:algebraization-FGC}, we introduce the Hilbert calculus, prove its soundness and completeness with respect to the variety of Frobenius-adjoint residuated lattices, and show that the resulting logic is algebraizable in the sense of Blok and Pigozzi. We further study this new class of residuated lattices by imposing suitable Frobenius-type compatibility conditions on a unary Galois connection and establish its algebraic properties.

Constructive logic with strong negation \cite{Rasiowa} was created for logical reasons: the
intuitionistic negation does not have good constructive properties. From the beginning, constructive logic with strong negation was closely related to intuitionistic logic. The algebraic counterpart of that logic, first studied by Rasiowa and later defined as a variety by Brignole and Monteiro \cite{Brignolea,Brignoleb}, is called Nelson algebras.  The algebraic approach of the relationship between intuitionism and constructive
logic with strong negation appears: a characterization of Nelson algebras as pairs of disjoint elements of Heyting algebras \cite{Vakarelov}. Cignoli \cite{Cignoli} shows that this characterization can be formalized from a categorical point of view as an adjunction between the categories of Heyting algebras and Nelson algebras, which restricted to centered Nelson algebras becomes an equivalence that is induced by the Kalman functor, motivated by a construction due to Kalman \cite{Kalman}, where the original Kalman construction applied to a bounded distributive lattice $\mathbf{L}$ gives the centered Kleene algebra \begin{center} $\mathbf{K}(\mathbf{L})=\{(x,y)\in L\times L \mid x\wedge y=0\}$.
\end{center} Motivated by \cite{Cignoli},  Castiglioni, Menni and Sagastume \cite {Castiglionib} replaced “$\wedge$” by “$\odot$” and applied the Kalman construction for a residuated distributive lattice $L$ by 
\begin{center} $\mathbf{K}(\mathbf{L})=\{(x,y)\in L\times L^o \mid x\odot y=0\}$,
\end{center} where $L^o$ is the dual order of $L$. This assignment $\mathbf{K}$ extends to a functor $\mathbf{K}:\mathbb{RDL}\rightarrow \mathbb{D\text{iff}RDL}$, where $\mathbb{RDL}$ is the category of residuated distributive lattices and $\mathbb{D\text{iff}RDL}$ is the category  of c-differential residuated distributive lattices, which actually are involutive residuated distributive lattices with a center $c$ satisfying the ``Leibniz'' condition, respectively. They also proved in \cite{Cignoli} that $\mathbf{K}$ has a left adjoint $\mathbf{C}:\mathbb{D\text{iff}RDL}\rightarrow \mathbb{RDL}$, showing that the adjunction $\mathbf{C}\dashv \mathbf{K}:\mathbb{RDL}\rightarrow \mathbb{D\text{iff}RDL}$ restricts to an equivalence $\mathbf{C}\dashv \mathbf{K}:\mathbb{RDL}\rightarrow \mathbb{D\text{iff}RDL'}$, where the category $\mathbb{D\text{iff}RDL'}$ is the full subcategory of $\mathbb{D\text{iff}RDL}$ whose objects $\boldsymbol{\mathfrak{L}}$ satisfy the condition $(\mathbf{CK})$, 
\begin{center}$(\mathbf{CK})$~ For any $x,y\in\mathcal{L}$ with $x,y\geq c,x\otimes y\leq c$, there exists $z\in \mathcal{L}$ such that $z\cup c=x,{\sim}z\cup c=y$.
\end{center}

The second motivation for our study stems from results due to Castiglioni et al. relating residuated distributive lattices and  c-differential residuated distributive lattices with $\mathbf{CK}$ by the functor $\mathbf{K}$, and the conviction that these can be lifted to the level of Frobenius-adjoint residuated distributive lattices. Let $\mathbb{FARDL}$ be the category of Frobenius-adjoint residuated distributive lattices. The main aim of our investigation is to isolate a reasonable category which completes the diagram:
\begin{center}
\begin{tikzpicture}
  \node[inner sep=2pt] (MV^bullet) at (0,0) {$\mathbb{RDL}$};
  \node[inner sep=2pt] (MV) at (0,3) {$\mathbb{FARDL}$ };
  \node[inner sep=2pt] (MDRL) at (3,0) {$\mathbb{D\text{iff}RDL'}$};
  \node[inner sep=2pt] (iIRL_0) at (3,3) {$\mathbb{?}$};  
\draw[->] ([yshift=3pt]MV.east) -- ([yshift=3pt]iIRL_0.west) node[midway, above=3pt] {$\mathbf{K}$};
\draw[<-] ([yshift=-3pt]MV.east) -- ([yshift=-3pt]iIRL_0.west) node[midway, below=3pt] {$\mathbf{C}$};
\draw[->] ([yshift=3pt]MDRL.west) -- ([yshift=3pt]MV^bullet.east) node[midway, above=3pt] {$\mathbf{K}$};
\draw[<-] ([yshift=-3pt]MDRL.west) -- ([yshift=-3pt]MV^bullet.east) node[midway, below=3pt] {$\mathbf{C}$};

\draw[->] (MV) -- (MV^bullet) node[midway, left=2pt] {$\mathbf{F}$};
  \draw[->] (iIRL_0) -- (MDRL) node[midway, right=2pt] {$\mathbf{F}$};
\end{tikzpicture} 
\begin{center}  {\bf{Figure 2.}} Diagram of the functors $\mathbf{K}$ and $\mathbf{C}$ relate the categories $\mathbb{FARDL}$ and $\mathbb{?}$ .
\end{center}
\end{center}
where $\mathbb{FARDL}\rightarrow \mathbb{RDL}$ is the obvious forgetful functor. The category completing this square should be a category of c-differential residuated distributive lattices with the condition $\mathbf{CK}$ equipped with a Frobenius-adjoint  structure in such a way that the two structures interact in a reasonable way. Note here that c-differential residuated distributive lattices satisfy double negation laws, hence in the definition of the corresponding Frobenius-adjoint algebras, it is possible to use only one of the Frobenius right and left adjoint operators as primitive, the other being definable as the dual of the one taken as part of the signature. However, the definition of Frobenius-adjoint residuated distributive lattices requires the introduction of two kinds of Frobenius-adjoint operators simultaneously, because these Frobenius right and left adjoint operators are not mutually interdefinable. It is clear that the structure of a Frobenius-adjoint c-differential residuated distributive lattice $(\boldsymbol{\mathfrak{L}},f_c)$ is simpler than that of a weak tense residuated distributive lattice $(\mathbf{L},f,g)$, because it has only one of the Frobenius right and left adjoint operators. In view of this point, the result of the above diagram shows that one can equivalently replace a complex category with a simple one, which will be of great benefit to study the algebraic structure of the complex one. 

Accordingly, the second aim of this paper is to complete the above categorical diagram. In Section~\ref{section4.}, we introduce an appropriate category of Frobenius-adjoint $c$-differential residuated distributive lattices satisfying $\mathbf{CK}$, and extend the Kalman functor $\mathbf{K}$ to the Frobenius-adjoint setting. We then prove that the resulting functors establish an equivalence between this category and $\mathbb{FARDL}$.

The third motivation for our study concerns the logical significance of the categorical representation developed above. The logical relevance of positive-cone constructions has already appeared in related settings. For instance, positive cones of suitable residuated structures have been used to provide algebraic semantics for paraconsistent Nelson logic \cite{BusanicheCignoli2009}. From a more general perspective, Moraschini \cite{Moraschini2018} showed that adjunctions between algebraic categories can be reflected at the level of relative equational consequence through suitable contextual translations. These results indicate that an algebraic or categorical representation may carry nontrivial logical information and, in particular, suggest that a categorical equivalence should admit a corresponding interpretation at the level of logical consequence. In the present setting, Section~\ref{section4.} establishes an equivalence between the category $\mathbb{FARDL}$ of Frobenius-adjoint residuated distributive lattices and the corresponding category of Frobenius-adjoint $c$-differential residuated distributive lattices satisfying $\mathbf{CK}$. Since every object in the latter category determines a Frobenius-adjoint residuated distributive lattice through its positive cone, it is natural to ask whether this representation also preserves the logical consequence relation associated with $\mathbb{FARDL}$. More precisely, one would like to know whether formulas can be interpreted directly in positive cones without changing their semantic consequence, and whether algebraic properties relevant to the logic, such as equational consequence, conservativity, and the existence of finite countermodels, are preserved under the Kalman equivalence. While the works mentioned above provide important precedents for the logical use of positive cones and for the relation between categorical adjunctions and consequence, a systematic positive-cone semantics of this kind has not been developed in the Frobenius-adjoint setting. Establishing such a connection therefore fills a natural gap between the categorical and logical aspects of the theory and constitutes one of the main contributions of the present paper.

Accordingly, the third aim of this paper is to develop the logical counterpart of the categorical equivalence established in Section~\ref{section4.}. In Section~\ref{sec:positive-cones}, we introduce the distributive logic $\mathbf{FL}^{\mathbf{FGC},d}_{ew}$ and establish its positive-cone semantics over Frobenius-adjoint $c$-differential residuated distributive lattices satisfying $\mathbf{CK}$, proving the equivalence between derivability $\mathbb{FARDL}$-semantics, and positive-cone semantics, and further examine equational consequence, conservativity, and finite countermodels under the Kalman equivalence.

The paper is organized as follows. Section~\ref{section:preliminaries-RL} recalls some basic notions and fundamental results on residuated lattices, $c$-differential residuated distributive lattices, and the Kalman construction. Section~\ref{section:algebraization-FGC} introduces the logic $\mathbf{FL}^{\mathbf{FGC}}_{ew}$, proves its algebraizability, and identifies Frobenius-adjoint residuated lattices as its equivalent algebraic semantics, together with some basic algebraic properties and finite-model properties. Section~\ref{section4.} extends the Kalman construction to the Frobenius-adjoint setting and establishes a categorical equivalence between Frobenius-adjoint residuated distributive lattices and Frobenius-adjoint $c$-differential residuated distributive lattices with the condition $\mathbf{CK}$. Section~\ref{sec:positive-cones} introduces the distributive extension $\mathbf{FL}^{\mathbf{FGC},d}_{ew}$ and develops its positive-cone semantics, relating logical consequence, equational consequence, conservativity, and finite countermodels to the categorical equivalence obtained in Section~\ref{section4.}.

\section{Preliminaries}\label{section:preliminaries-RL}

For the reader's convenience, we include a table summarizing the notation used for classes of algebras throughout the paper, followed by the corresponding definitions.

\begin{table}[H]
\footnotesize
\begin{center}
\begin{tabular}{ccc}
\midrule 
\textbf{Class of algebras} &\textbf{Signature} &\textbf{Notation} \\
\midrule
Class of residuated lattices & $(L,\wedge,\vee,\odot,\rightarrow,0,1)$ & $\mathbb{RL}$ \\
Class of residuated distributive lattices & $(L,\wedge,\vee,\odot,\rightarrow,0,1)$ & $\mathbb {RDL}$ \\
Class of c-differential residuated distributive lattices & $(\mathcal{L},\cap,\cup,\otimes,{\sim},c,0,1)$ & $\mathbb{D}\text{iff}\mathbb{RDL}$\\
Class of  c-differential residuated distributive lattices with $\mathbf{CK}$  & $(\mathcal{L},\cap,\cup,\otimes,{\sim},c,0,1)$ & $\mathbb{D}\text{iff}\mathbb{RDL'}$\\
Class of Frobenius-adjoint residuated lattices & $(L,\wedge,\vee,\odot,\rightarrow,f,g,0,1)$ & $\mathbb{FARL}$ \\
Class  of Frobenius-adjoint involutive residuated lattices & $(L,\wedge,\vee,\odot,\rightarrow,f,0,1)$ &  $\mathbb{FAIRL}$ \\
Class of Frobenius-adjoint residuated distributive lattices & $(L,\wedge,\vee,\odot,\rightarrow,f,g,0,1)$ & $\mathbb{FARDL}$ \\
Class of Frobenius-adjoint c-differential residuated distributive lattices & $(\mathcal{L},\cap,\cup,\otimes,\sim,f_c,0,c,1)$ & $\mathbb{FA}\mathbb{D}\text{iff}\mathbb{RDL}$\\
Class of Frobenius-adjoint c-differential residuated distributive lattices with $\mathbf{CK}$  & $(\mathcal{L},\cap,\cup,\otimes,\sim, f_c,0,c,1)$ & $\mathbb{FA}\mathbb{D}\text{iff}\mathbb{RDL'}$\\
\bottomrule
\end{tabular}
\end{center}
\centering
{\bf{Table 1.}}  Notations and Signatures for the class of algebras covered throughout the paper.
\end{table}

In this section, we summarize some basic results concerning the categories $\mathbb{RDL}$, whose objects are residuated distributive lattices, and $\mathbb{D\text{iff}RDL}$ , whose objects are c-differential residuated distributive lattices. The category $\mathbb{D\text{iff}RDL'}$, whose objects are c-differential residuated lattices satisfying the condition $\mathbf{CK}$, arises as the image of $\mathbb{RDL}$ under the Kalman functor $\mathbf{K}$, which establishes the categorical equivalence between them. 
\begin{definition}\label{definietion2.1}\cite{Galatos} An algebraic structure $\mathbf{L}=(L,\wedge,\vee,\odot,\rightarrow,0,1)$ of type $(2,2,2,2,0,0)$ is called a \emph{residuated lattice\footnote{In \cite{Galatos}, the term “residuated lattice” is some times used for more general structures than those considered here. In broader terminology, the structures we define here would be called bounded integral commutative residuated lattices.}} if it satisfies the following conditions:

  (1)\ $(L,\wedge,\vee,0,1)$ is a bounded lattice,

  (2)\ $(L,\odot,1)$ is a  commutative monoid,

  (3)\ $x\odot y\leq z$ iff $x\leq y\rightarrow z$ for any $x,y,z\in L$.
  \end{definition}

   Clearly, the class $\mathbb{RL}$ forms a variety. In what follows, we denote by $\mathbf{L}$ the universe of a residuated lattice $(L,\wedge,\vee,\odot,\rightarrow,0,1)$. In any residuated lattice $\mathbf{L}$, we define other operations
\begin{center}$x^0=1$ and $x^n=x^{n-1}\odot x$  for $n\geq 1$, 

${\neg}x=x\rightarrow 0$ and ${\neg\neg} x=\neg(\neg x)$ for any $x,y\in L$.
\end{center}

  \begin{definition}\label{definietion2.2} \cite{Galatos} A residuated lattice $\mathbf{L}$ is

  (1)\ \emph{distributive} iff its underlying lattice is distributive.

  (2)\ \emph{involutive} iff it satisfies (INV), where
  \begin{center} (INV)~~~$\neg\neg x=x$ 
  \end{center} for any $x\in L$.
  
  \end{definition}
  
Here we recall some results concerning the equivalence of the categories $\mathbb{RDL}$ and $\mathbb{D\text{iff}RDL'}$. To distinguish the objects in $\mathbb{RDL}$ and $\mathbb{D\text{iff}RDL'}$, we use $\boldsymbol{\mathfrak{L}}$ to denote the universe of the involutive residuated lattice in $\mathbb{D\text{iff}RDL'}$. For more exposition on this subject, we refer the reader to \cite{Galatosb,Castiglionib,Castiglionic}.

\begin{definition}\label{definition2.3.}\cite{Galatosb} An algebraic structure $\boldsymbol{\mathfrak{L}}=(\mathcal{L},\cap,\cup,\otimes,{\sim},0,1)$ is an \emph{involutive residuated lattice} if it satisfies the following conditions:

(1)\ $(\mathcal{L},\cap,\cup,0,1)$ is a bounded lattice,

(2)\ $(\mathcal{L},\otimes,1)$ is a commutative monoid,

(3)\ ${\sim}$ is an involution of the lattice that is a dual automorphism,

(4)\ $x\otimes y\leq z$ iff $x\leq {\sim}(y\otimes {\sim} z)$ for any $x,y,z\in \mathcal{L}$.
\end{definition}
In what follows, we denote by $\boldsymbol{\mathfrak{L}}$ the universe of an involutive residuated lattice $(\mathcal{L},\cap,\cup,\otimes,{\sim},0,1)$.

\begin{remark}\label{remark2.4.} Note that involutive residuated lattices and dualizing residuated lattices are equivalent (see  \cite{Castiglionib}, Remark 4.1). Recalling in \cite{Castiglionib} that an algebra $(\mathcal{L},\cap,\cup,\otimes,\leadsto,0,1)$ is said to be a \emph{dualizing residuated lattice} provided it satisfies the following conditions:

(1) $(\mathcal{L},\cap,\cup,\otimes,\leadsto,0,1)$ is a residuated lattice,

(2) $0$ is a cyclic dualizing element. That is
$(x\leadsto 0)\leadsto 0=x$, for all $x\in \mathcal{L}$.
\end{remark}

An involutive residuated lattice $\boldsymbol{\mathfrak{L}}$ is called \emph{centered} if it contains a distinguished element, called a \emph{center}, which is a fixed point of the involution. A \emph{c-differential residuated lattice} is an involutive residuated lattice $\boldsymbol{\mathfrak{L}}$ with center $c$ satisfying the ``Leibniz''
 condition (see \cite{Castiglionib}, Definition 7.2):
\begin{center} 
$(x \otimes y) \cap c = ((x \cap c) \otimes y) \cup (x \otimes (y \cap c))$,
\end{center}
for all $x,y \in \boldsymbol{\mathfrak{L}}$. A c-differential residuated lattice is said to be \emph{distributive} if its underlying lattice is a bounded distributive lattice. \medskip

Castiglioni et al. constructed in \cite{Castiglionib} a c-differential residuated distributive lattice from a residuated distributive lattice. In what follows, we provide a summary of this construction.\medskip

Let $\mathbf{L}\in \mathbb{RDL}$. Define $\mathbf{K}(L)$ in the following way:
\begin{center} $\mathbf{K}(L)=\{(x,y)\in L\times L^o \mid x\odot y=0\}$,
\end{center}
where $L^o$ is the dual order of $L$ (see \cite{Castiglionib}, Section 7). Define the operations in $\mathbf{K}(L)$:
\begin{center} $(x,y)\cup (z,w)=(x\vee z,y\wedge w)$,

$(x,y)\cap (z,w)=(x\wedge z,y\vee w)$,

$(x,y)\otimes(z,w)=(x\odot z,(x\rightarrow w)\wedge (z\rightarrow y))$,

$(x,y)\leadsto(z,w)=((x\rightarrow z)\wedge (w \rightarrow y), w\odot x)$.

\end{center}
Then the algebra $\mathbf{K(L)}=(\mathbf{K}(L),\cap,\cup,\otimes,\leadsto,(0,1),(1,0))$ is a dualizing residuated distributive lattice, and $(1,1)$ is a cyclic dualizing element, the induced order $\leq$ by $\cap$ in $\mathbf{K(L)}$ is given by
 \begin{center} $(x_1,y_1)\leq (x_2,y_2)$ if and only if $x_1\leq x_2$ and $y_1\geq y_2$
 \end{center}
 (see \cite{Tsinakis}, Corollary 3.5). The algebraic structure $(\mathbf{K}(L),\cap,\cup,\otimes,{\sim},(0,1),(1,0)),$ with 
\begin{center} ${\sim}(x,y)=(x,y)\leadsto (1,1)=(y,x)$
 \end{center}
is an involutive residuated lattice with $c=(0,0)$ (see  \cite{Castiglionib}, Remark 4.1), and an object of $\mathbb{D\text{iff}RDL}$. The assignment $\mathbf{K}$ extends to a functor $\mathbf{K}:\mathbb{RDL}\rightarrow \mathbb{D\text{iff}RDL}$ (see \cite{Castiglionib}, Lemma 7.1).

\medskip

Let $\boldsymbol{\mathfrak{L}}\in \mathbb{D\text{iff}RDL}$ and $\mathcal{C}(\mathcal{L})=\{x\in \mathcal{L}\mid x\geq c\}$. In  $\mathcal{C}(\mathcal{L})$ define the product
\begin{center} $x\otimes_cy=(x\otimes y)\cup c$,
\end{center}
the bottom as the constant $c$ and the other operations as those induced from $\boldsymbol{\mathfrak{L}}$ with ${\sim}x_c=x\leadsto c$. Then $\mathcal{C}(\boldsymbol{\mathfrak{L}})=(\mathcal{C}(\mathcal{L}),\cap,\cup,\otimes_c,\leadsto,c,1)\in\mathbb{D\text{iff}RDL}$ and $\boldsymbol{\mathfrak{L}}\mapsto \mathcal{C}(\boldsymbol{\mathfrak{L}})$ defines another functor
\begin{center} $\mathbf{C}:\mathbb{D\text{iff}RDL}\rightarrow\mathbb{RDL}$ ,
\end{center}
which is left adjoint to $\mathbf{K}$ (see \cite{Castiglionib}, Theorem 7.6).

The adjunction
\begin{center} $\mathbf{C}\dashv \mathbf{K}:\mathbb{RDL}\rightarrow \mathbb{D\text{iff}RDL}$
 \end{center}
restricts to an equivalence $\mathbf{C}\dashv \mathbf{K}:\mathbb{RDL}\rightarrow \mathbb{D\text{iff}RDL'}$ (see \cite{Castiglionib}, Corollary 7.8),
where $\mathbb{D\text{iff}RDL'}$, is the subcategory of the category $\mathbb{D\text{iff}RDL}$, whose objects $\boldsymbol{\mathfrak{L}}$ satisfy the following condition:
\begin{center} $(\mathbf{CK})$ \ ~~~For every pair of element $x,y\in \mathcal{L}$ such that $x,y\geq c$ and $x\otimes y\leq c$, there exists $z\in \mathcal{L}$ such that $z\cup c=x$ and ${\sim}z\cup c=y$.
\end{center}

Castiglioni et al. stated in \cite{Castiglionic} that $\mathbb{D\text{iff}RDL'}$ forms a variety. But, strictly speaking, $\mathbb{D\text{iff}RDL'}$ does not forms a variety. It is the class of all $\kappa$-free reducts of the variety that Castiglioni et al. defined. In precisely, let $\boldsymbol{\mathfrak{L}}\in \mathbb{D\text{iff}RDL'}$. Then there is a map $\kappa:\mathcal{L}\rightarrow \mathcal{L}$ that satisfies the conditions:

\begin{center}

$\kappa x\cup c=c\leadsto x$~~and~~$\kappa x\cup c=x\cup c$.
\end{center}

Conversely, if $\boldsymbol{\mathfrak{L}} \in \mathbb{D\text{iff}RDL'}$ in which there exists an operator $\kappa$ that satisfies the previous equations, then $(\mathbf{CK})$ holds on $\boldsymbol{\mathfrak{L}}$ (see \cite{Castiglionic}, Theorem 1).

\section{Algebraization of substructural logics with Frobenius--Galois connections}\label{section:algebraization-FGC}
\label{section3}

Inspired by an intuitionistic propositional logic with a Galois connection $\mathbf{IntGC}$ \cite{DzikJarvinenKondo2010}, we enlarge the language of $\mathbf{FL}_{ew}$ by introducing two unary connectives $\blacktriangle$ and $\triangledown$. The resulting logic is called  \emph{the substructural logic $\mathbf{FL_{ew}}$ with a Frobenius-Galois connection} and denoted by $\mathbf{FL^{FGC}}_{ew}$. In this section, we first define the Hilbert calculus $\mathbf{FL^{FGC}}_{ew}$ and prove that it is algebraizable, and then identify its equivalent algebraic semantics with the variety of Frobenius-adjoint residuated lattices, and prove conservativity over $\mathbf{FL}_{ew}$, discussing the finite model property of $\mathbf{FL^{FGC}}_{ew}$.

\subsection{Algebraizability of $\mathbf{FL}^{\mathrm{FGC}}_{ew}$}
\label{subsection:algebraizability-FGC}

In this subsection, we introduce the Hilbert-style calculus $\mathbf{FL^{FGC}}_{ew}$ and establish its soundness and completeness with respect to the intended algebraic semantics, further proving that it is algebraizable in the sense of Blok and Pigozzi \cite{BlokPigozzi}.
\medskip

Let
\[
 \mathcal L_{\mathrm{FGC}}
   =\{\sqcap,\sqcup,\&,\Rightarrow,\bot,\top,\triangledown,\blacktriangle\}
\]
be the propositional language obtained from the language of
$\mathbf{FL}_{ew}$ by adding the unary connectives $\triangledown$ and
$\blacktriangle$.  We write $\mathrm{Fm}_{\mathcal L_{\mathrm{FGC}}}$ for the
absolutely free $\mathcal L_{\mathrm{FGC}}$-algebra over a fixed countably
infinite set of propositional variables.  For formulas $\varphi$ and $\psi$, we
use the abbreviation
\begin{center}
 $\varphi\Leftrightarrow\psi
   :=(\varphi\Rightarrow\psi)\sqcap(\psi\Rightarrow\varphi)$.
\end{center}

\begin{definition}\label{def:logic-FGC}
The logic $\mathbf{FL}^{\mathrm{FGC}}_{ew}$ is the expansion of
$\mathbf{FL}_{ew}$ obtained by adding the axiom schemes:
\begin{align}
 \triangledown\varphi&\Rightarrow\varphi,
   \tag{FGC1}\label{ax:FGC1}\\
 \triangledown(\blacktriangle\varphi\Rightarrow\psi)
   &\Leftrightarrow
   (\blacktriangle\varphi\Rightarrow\triangledown\psi),
   \tag{FGC2}\label{ax:FGC2}\\
 \triangledown(\blacktriangle\varphi\sqcup\psi)
   &\Leftrightarrow
   (\blacktriangle\varphi\sqcup\triangledown\psi),
   \tag{FGC3}\label{ax:FGC3}\\
 \blacktriangle(\varphi\sqcap\blacktriangle\psi)
   &\Leftrightarrow
   (\blacktriangle\varphi\sqcap\blacktriangle\psi).
   \tag{FGC4}\label{ax:FGC4}
\end{align}
In addition to modus ponens
\[
 (\mathrm{MP})\qquad
 \frac{\varphi\qquad\varphi\Rightarrow\psi}{\psi},
\]
two Galois rules
\[
 (\mathrm{GC}_1)\qquad
 \frac{\blacktriangle\varphi\Rightarrow\psi}
      {\varphi\Rightarrow\triangledown\psi},
 \qquad
 (\mathrm{GC}_2)\qquad
 \frac{\varphi\Rightarrow\triangledown\psi}
      {\blacktriangle\varphi\Rightarrow\psi},
\]
and the necessitation rule
\[
 (\mathrm N_{\triangledown})\qquad
 \frac{\varphi}{\triangledown\varphi}.
\]

For $\Gamma\cup\{\varphi\}\subseteq
\mathrm{Fm}_{\mathcal L_{\mathrm{FGC}}}$, a \emph{derivation} of $\varphi$ from
$\Gamma$ is a finite sequence of formulas whose last member is $\varphi$ and
each of whose members is either an element of $\Gamma$, an instance of an axiom
of $\mathbf{FL}_{ew}$, an instance of \eqref{ax:FGC1}--\eqref{ax:FGC4}, or is
obtained from members by one of the rules displayed above.  We write
\[
 \Gamma\vdash_{\mathbf{FL}^{\mathrm{FGC}}_{ew}}\varphi
\]
when such a derivation exists.  A formula $\varphi$ is a \emph{theorem} when
$\varnothing\vdash_{\mathbf{FL}^{\mathrm{FGC}}_{ew}}\varphi$, and denoted by
$\vdash_{\mathbf{FL}^{\mathrm{FGC}}_{ew}}\varphi$.
\end{definition}

A set $T\subseteq\mathrm{Fm}_{\mathcal L_{\mathrm{FGC}}}$ is called a
\emph{theory} of $\mathbf{FL}^{\mathrm{FGC}}_{ew}$, whenever $T\vdash_{\mathbf{FL}^{\mathrm{FGC}}_{ew}}\varphi$,
then $\varphi\in T$.

\begin{proposition}\label{prop:FGC-monotonicity}
Let $T$ be a theory of $\mathbf{FL}^{\mathrm{FGC}}_{ew}$ and
$\varphi,\psi\in\mathrm{Fm}_{\mathcal L_{\mathrm{FGC}}}$.

 (1)~If
 $T\vdash_{\mathbf{FL}^{\mathrm{FGC}}_{ew}}\varphi\Rightarrow\psi$, then
 $T\vdash_{\mathbf{FL}^{\mathrm{FGC}}_{ew}}
   \triangledown\varphi\Rightarrow\triangledown\psi
 ~\text{and}~
 T\vdash_{\mathbf{FL}^{\mathrm{FGC}}_{ew}}
   \blacktriangle\varphi\Rightarrow\blacktriangle\psi$,

 (2)~If
 $T\vdash_{\mathbf{FL}^{\mathrm{FGC}}_{ew}}
   \varphi\Leftrightarrow\psi$, then $T\vdash_{\mathbf{FL}^{\mathrm{FGC}}_{ew}}
   \triangledown\varphi\Leftrightarrow\triangledown\psi
 ~\text{and}~
 T\vdash_{\mathbf{FL}^{\mathrm{FGC}}_{ew}}
   \blacktriangle\varphi\Leftrightarrow\blacktriangle\psi$.
\end{proposition}

\begin{proof}
The reflexivity theorem
$\triangledown\varphi\Rightarrow\triangledown\varphi$, followed by
$(\mathrm{GC}_2)$ with $\varphi:=\triangledown\varphi$ and
$\psi:=\varphi$, gives
\[
 T\vdash_{\mathbf{FL}^{\mathrm{FGC}}_{ew}}
 \blacktriangle\triangledown\varphi\Rightarrow\varphi.
\]

If $T\vdash_{\mathbf{FL}^{\mathrm{FGC}}_{ew}}\varphi\Rightarrow\psi$, then by the transitivity of implication in $\mathbf{FL}_{ew}$,
\[
 T\vdash_{\mathbf{FL}^{\mathrm{FGC}}_{ew}}
 \blacktriangle\triangledown\varphi\Rightarrow\psi.
\]

Applying $(\mathrm{GC}_1)$ gives
\[
 T\vdash_{\mathbf{FL}^{\mathrm{FGC}}_{ew}}
 \triangledown\varphi\Rightarrow\triangledown\psi.
\]

For the second unary connective, start with the reflexivity theorem
$\blacktriangle\psi\Rightarrow\blacktriangle\psi$.  Applying $(\mathrm{GC}_1)$, with
$\varphi:=\psi$ and $\psi:=\blacktriangle\psi$, gives
\[
 T\vdash_{\mathbf{FL}^{\mathrm{FGC}}_{ew}}
 \psi\Rightarrow\triangledown\blacktriangle\psi.
\]

Combining this formula with $T\vdash_{\mathbf{FL}^{\mathrm{FGC}}_{ew}}\varphi\varphi\Rightarrow\psi$ yields
\[
 T\vdash_{\mathbf{FL}^{\mathrm{FGC}}_{ew}}
 \varphi\Rightarrow\triangledown\blacktriangle\psi.
\]

Applying $(\mathrm{GC}_2)$ now gives
\[
 T\vdash_{\mathbf{FL}^{\mathrm{FGC}}_{ew}}
 \blacktriangle\varphi\Rightarrow\blacktriangle\psi.
\]

(2) follows by applying (1) to the two
implications occurring in the definition of
$\varphi\Leftrightarrow\psi$.
\end{proof}

Notice that $\mathbf{FL}_{ew}$ is algebraizable with defining equation
$p\approx1$ and equivalence formula $p\Leftrightarrow q$.  We now verify that
the same translations algebraize the present expansion logic $\mathbf{FL}^{\mathrm{FGC}}_{ew}$.  Define
\begin{equation}
 \tau(\varphi)=\{\varphi\approx1\},
 \qquad
 \rho(s\approx t)=\{s\Leftrightarrow t\}.
 \label{eq:algebraizing-translations}
\end{equation}
In an algebraic interpretation, $\triangledown$ is represented by a unary
operator $f$ and $\blacktriangle$ by a unary operator $g$. 

 Accordingly, the
algebraic translations of \eqref{ax:FGC1}--\eqref{ax:FGC4} are
\begin{align}
 f(x)&\le x,
   \label{eq:KFGC1}\\
 f(g(x)\to y)&=g(x)\to f(y),
   \label{eq:KFGC2}\\
 f(g(x)\vee y)&=g(x)\vee f(y),
   \label{eq:KFGC3}\\
 g(x\wedge g(y))&=g(x)\wedge g(y).
   \label{eq:KFGC4}
\end{align}

The two Galois rules translate into the quasi-equations
\begin{align}
 g(x)\le y&\Longrightarrow x\le f(y),
   \label{eq:KGC1}\\
 x\le f(y)&\Longrightarrow g(x)\le y.
   \label{eq:KGC2}
\end{align}

Let $\mathbb K_{\mathrm{FGC}}$ denote the class of all algebras
\[
 \mathscr{A}=(\mathbf{A},f,g)=(A,\wedge,\vee,\odot,\to,f,g,0,1)
\]
whose residuated-lattice reduct belongs to $\mathbb{RL}$ and which satisfy
\eqref{eq:KFGC1}--\eqref{eq:KGC2}.  

For
$\Gamma\cup\{\varphi\}\subseteq
\mathrm{Fm}_{\mathcal L_{\mathrm{FGC}}}$, we write
\[
 \Gamma\models_{\mathbb K_{\mathrm{FGC}}}\varphi
\]
when, for every $\mathscr{A}\in\mathbb K_{\mathrm{FGC}}$ and every valuation
$v:\mathrm{Fm}_{\mathcal L_{\mathrm{FGC}}}\to\mathscr{A}$, the condition
$v(\gamma)=1$ for all $\gamma\in\Gamma$ implies $v(\varphi)=1$.

\begin{theorem}\label{thm:FGC-soundness}
For every $\Gamma\cup\{\varphi\}\subseteq
\mathrm{Fm}_{\mathcal L_{\mathrm{FGC}}}$,
\[
 \Gamma\vdash_{\mathbf{FL}^{\mathrm{FGC}}_{ew}}\varphi
 \Longrightarrow
 \Gamma\models_{\mathbb K_{\mathrm{FGC}}}\varphi.
\]
\end{theorem}

\begin{proof}
Soundness of the $\mathbf{FL}_{ew}$ fragment is standard.  Let
$\mathscr{A}\in\mathbb K_{\mathrm{FGC}}$ and let $v$ be a valuation into
$\mathscr{A}$.  Equation \eqref{eq:KFGC1} gives
$v(\triangledown\varphi)\le v(\varphi)$, so every instance of
\eqref{ax:FGC1} has value $1$.  Equations
\eqref{eq:KFGC2}--\eqref{eq:KFGC4} show that the two sides of each of
\eqref{ax:FGC2}--\eqref{ax:FGC4} have same value.   Note that in any residuated lattice,
\begin{center} $a=b  \Longleftrightarrow
(a\to b)\wedge(b\to a)=1$, 
\end{center}
these axiom schemes are valid.

For $(\mathrm{GC}_1)$, assume
$v(\blacktriangle\varphi\to\psi)=1$.  Then
$g(v(\varphi))\le v(\psi)$, and \eqref{eq:KGC1} gives
$v(\varphi)\le f(v(\psi))$.  Hence
$v(\varphi\to\triangledown\psi)=1$.  The same argument using
\eqref{eq:KGC2} proves soundness of $(\mathrm{GC}_2)$.

It remains to check necessitation.  Since $g(1)\le1$, condition
\eqref{eq:KGC1} yields $1\le f(1)$; hence $f(1)=1$.  Therefore
$v(\varphi)=1$ implies
$v(\triangledown\varphi)=f(1)=1$.  Modus ponens and substitution preserve
validity, so induction on derivations proves the result.
\end{proof}

\begin{lemma}\label{lem:FGC-Lindenbaum-congruence}
Let $T$ be a theory of $\mathbf{FL}^{\mathrm{FGC}}_{ew}$ and define
\[
 \varphi\equiv_T\psi
 \Longleftrightarrow
 T\vdash_{\mathbf{FL}^{\mathrm{FGC}}_{ew}}
   \varphi\Leftrightarrow\psi.
\]
Then $\equiv_T$ is a congruence on
$\mathrm{Fm}_{\mathcal L_{\mathrm{FGC}}}$.
\end{lemma}

\begin{proof}
Reflexivity, symmetry, and transitivity follow from the standard properties of
$\Leftrightarrow$ in $\mathbf{FL}_{ew}$.  Compatibility with
$\sqcap,\sqcup,\&,\Rightarrow,\bot,\top$ is inherited from the congruentiality of $\mathbf{FL}_{ew}$.  Compatibility with $\triangledown$ and $\blacktriangle$ is exactly
Proposition~\ref{prop:FGC-monotonicity}(2).  Hence every operation
of $\mathcal L_{\mathrm{FGC}}$ preserves $\equiv_T$.
\end{proof}

\begin{lemma}\label{lem:FGC-canonical-algebra}
Let $T$ be a theory of $\mathbf{FL}^{\mathrm{FGC}}_{ew}$ and put
\[
 \mathscr{A}_T
   =\mathrm{Fm}_{\mathcal L_{\mathrm{FGC}}}/\!\equiv_T.
\]
Then $\mathscr{A}_T\in\mathbb K_{\mathrm{FGC}}$.  Moreover, if
$[\alpha]_T$ denotes the $\equiv_T$-class of $\alpha$, then
\begin{equation}
 [\alpha]_T\le[\beta]_T
 \Longleftrightarrow
 T\vdash_{\mathbf{FL}^{\mathrm{FGC}}_{ew}}\alpha\Rightarrow\beta.
 \label{eq:canonical-order}
\end{equation}
\end{lemma}

\begin{proof}
By Lemma~\ref{lem:FGC-Lindenbaum-congruence}, all operations are well defined on
the quotient.  Its reduct to the language of $\mathbf{FL}_{ew}$ is the usual
Lindenbaum--Tarski algebra of $T$ and therefore belongs to $\mathbb{RL}$.  The
standard description of the lattice order gives \eqref{eq:canonical-order}.

\eqref{ax:FGC1} gives
$[\triangledown\alpha]_T\le[\alpha]_T$, and
\eqref{ax:FGC2}--\eqref{ax:FGC4} yield, respectively,
\eqref{eq:KFGC2}--\eqref{eq:KFGC4} in the quotient.

Suppose that $[\blacktriangle\alpha]_T\le[\beta]_T$.  By
\eqref{eq:canonical-order},
$T\vdash_{\mathbf{FL}^{\mathrm{FGC}}_{ew}}
\blacktriangle\alpha\Rightarrow\beta$.  Rule $(\mathrm{GC}_1)$ gives
$T\vdash_{\mathbf{FL}^{\mathrm{FGC}}_{ew}}
\alpha\Rightarrow\triangledown\beta$, and hence
$[\alpha]_T\le[\triangledown\beta]_T$.  This proves
\eqref{eq:KGC1}.  

Conversely, if
$[\alpha]_T\le[\triangledown\beta]_T$, then
$T\vdash_{\mathbf{FL}^{\mathrm{FGC}}_{ew}}
\alpha\Rightarrow\triangledown\beta$.  Rule $(\mathrm{GC}_2)$ gives
$T\vdash_{\mathbf{FL}^{\mathrm{FGC}}_{ew}}
\blacktriangle\alpha\Rightarrow\beta$, whence
$[\blacktriangle\alpha]_T\le[\beta]_T$.  Thus \eqref{eq:KGC2} also holds.
Therefore $\mathscr{A}_T\in\mathbb K_{\mathrm{FGC}}$.
\end{proof}

\begin{theorem}\label{thm:FGC-completeness}
For every $\Gamma\cup\{\varphi\}\subseteq
\mathrm{Fm}_{\mathcal L_{\mathrm{FGC}}}$,
\[
 \Gamma\vdash_{\mathbf{FL}^{\mathrm{FGC}}_{ew}}\varphi
 \Longleftrightarrow
 \Gamma\models_{\mathbb K_{\mathrm{FGC}}}\varphi.
\]
\end{theorem}

\begin{proof}
The left-to-right implication is Theorem~\ref{thm:FGC-soundness}.  Conversely,
suppose that
$\Gamma\nvdash_{\mathbf{FL}^{\mathrm{FGC}}_{ew}}\varphi$ and let
\[
 T=\mathrm{Cn}_{\mathbf{FL}^{\mathrm{FGC}}_{ew}}(\Gamma)
   :=\{\psi:\Gamma\vdash_{\mathbf{FL}^{\mathrm{FGC}}_{ew}}\psi\}.
\]
Then $T$ is a theory, $\Gamma\subseteq T$, and $\varphi\notin T$.  By
Lemma~\ref{lem:FGC-canonical-algebra},
$\mathscr{A}_T\in\mathbb K_{\mathrm{FGC}}$.  Define the canonical valuation by
$v_T(p)=[p]_T$ for each propositional variable $p$.  Induction on formula
complexity gives
\[
 v_T(\alpha)=[\alpha]_T
 \qquad
 (\alpha\in\mathrm{Fm}_{\mathcal L_{\mathrm{FGC}}}).
\]

If $\gamma\in T$, then $T\vdash\gamma$.  Since in $\mathbf{FL}_{ew}$ a
formula is interderivable with its biconditional to $1$, it follows that
$T\vdash\gamma\Leftrightarrow1$, and hence $v_T(\gamma)=1$.  In particular,
$v_T(\gamma)=1$ for every $\gamma\in\Gamma$.

On the other hand, $v_T(\varphi)\ne1$.  Indeed, if
$[\varphi]_T=[1]_T$, then
$T\vdash\varphi\Leftrightarrow1$, so
$T\vdash1\Rightarrow\varphi$.  Modus ponens with the theorem $1$ would give
$T\vdash\varphi$, contrary to $\varphi\notin T$.  

Thus $\mathscr{A}_T$ and
$v_T$ form a countermodel, proving completeness.
\end{proof}

\begin{theorem}\label{thm:FGC-algebraizability}
The logic $\mathbf{FL}^{\mathrm{FGC}}_{ew}$ is algebraizable in the sense of
Blok and Pigozzi.  The maps $\tau$ and $\rho$ defined in
\eqref{eq:algebraizing-translations} form an algebraizing pair, and
$\mathbb K_{\mathrm{FGC}}$ is its equivalent algebraic semantics.
\end{theorem}

\begin{proof}
Theorem~\ref{thm:FGC-completeness} gives, for every
$\Gamma\cup\{\varphi\}\subseteq
\mathrm{Fm}_{\mathcal L_{\mathrm{FGC}}}$,
\[
 \Gamma\vdash_{\mathbf{FL}^{\mathrm{FGC}}_{ew}}\varphi
 \Longleftrightarrow
 \{\gamma\approx1:\gamma\in\Gamma\}
 \models_{\mathbb K_{\mathrm{FGC}}}\varphi\approx1.
\]

It remains to verify recovery under the two translations.

For every formula $\varphi$,
\begin{equation}
 \varphi
 \dashv\vdash_{\mathbf{FL}^{\mathrm{FGC}}_{ew}}
 \rho\tau(\varphi)
 =\varphi\Leftrightarrow1.
 \label{eq:formula-recovery}
\end{equation}
Indeed, $\varphi$ entails $1\Rightarrow\varphi$, while
$\varphi\Rightarrow1$ is a theorem of $\mathbf{FL}_{ew}$, and hence
$\varphi\vdash\varphi\Leftrightarrow 1$.  

Conversely,
$\varphi\Leftrightarrow1$ entails $1\Rightarrow\varphi$, and modus ponens with the
theorem $1$ yields $\varphi$.

For terms $s$ and $t$, every algebra in
$\mathbb K_{\mathrm{FGC}}$ has a residuated-lattice reduct, and hence
\begin{equation}
 s\approx t
 \Longleftrightarrow
 (s\Leftrightarrow t)\approx1.
 \label{eq:equation-recovery}
\end{equation}
The right-hand side is equivalent to the conjunction of $s\le t$ and
$t\le s$.  Thus every equation is equivalent over
$\mathbb K_{\mathrm{FGC}}$ to its round-trip translation
$\tau\rho(s\approx t)$.  Together with
Lemma~\ref{lem:FGC-Lindenbaum-congruence}, these facts are precisely the
Blok--Pigozzi algebraizability conditions.
\end{proof}

\subsection{The equivalent algebraic semantics of
$\mathbf{FL}^{\mathrm{FGC}}_{ew}$}
\label{subsection:FARL-semantics}

In this subsection, we introduce the variety $\FARL$ of Frobenius-adjoint residuated lattices and prove that is an equivalent algebraic semantics of $\mathbf{FL}^{\mathrm{FGC}}_{ew}$, giving an intrinsic description of $\mathbb K_{\mathrm{FGC}}$.  

\begin{definition}\label{def:FARL}
Let
$\mathbf{L}$ be a residuated lattice and
$f,g:L\to L$ be two unary operators.  The algebra
\[
 \mathscr{L}=
 (\mathbf{L},f,g) =(L,\wedge,\vee,\odot,\to,0,1,f,g)
\]
is called a \emph{Frobenius-adjoint residuated lattice} if, for all
$x,y\in L$, it satisfies
\begin{align}
 f(x)&\le x,
   \tag{FARL1}\label{FARL1}\\
 g(x)\le y&\Longleftrightarrow x\le f(y),
   \tag{FARL2}\label{FARL2}\\
 f(g(x)\to y)&=g(x)\to f(y),
   \tag{FARL3}\label{FARL3}\\
 f(g(x)\vee y)&=g(x)\vee f(y),
   \tag{FARL4}\label{FARL4}\\
 g(x\wedge g(y))&=g(x)\wedge g(y).
   \tag{FARL5}\label{FARL5}
\end{align}
The class of all Frobenius-adjoint residuated lattices is denoted by
$\FARL$.
\end{definition}

 \eqref{FARL2} states that $g$ is left adjoint to $f$, written
$g\dashv f$.  Accordingly, $g$ is called the \emph{Frobenius left adjoint operator} and
$f$ the \emph{Frobenius right adjoint operator} on the underlying residuated lattice.

\begin{remark}\label{rem:FARL-terminology}
The term \emph{Frobenius-adjoint} emphasizes four independent structural
features.  

(1)~~\eqref{FARL2} is the {\bf Galois adjunction}
\[
 g(x)\le y\Longleftrightarrow x\le f(y).
\]

(2)~~\eqref{FARL3} is the residual counterpart of the {\bf Frobenius-type reciprocity law}:
\[
 f(g(x)\to y)=g(x)\to f(y),
\]
and links the unary adjunction $g\dashv f$ with the residuated adjunction
$x\odot{-}\dashv x\to{-}$. 

(3)~~\eqref{FARL4} is a {\bf Frobenius-type reciprocity law for the join}:
\[
 f(g(x)\vee y)=g(x)\vee f(y).
\]
It says that an element in image of the left adjoint $g$ be moved
through the right adjoint $f$ w.r.t. $\vee$.  

(4)~~\eqref{FARL5} is a {\bf meet-form Frobenius law for the left adjoint $g$}:

\[g(x\wedge g(y))=g(x)\wedge g(y).
\]

Therefore, the name emphasizes both the adjoint relationship between $g$ and $f$ and their Frobenius-type interaction with the operations of the residuated lattice. 
\end{remark}

\begin{proposition}\label{prop:identify-KFGC-FARL}
The algebraic translation class $\mathbb K_{\mathrm{FGC}}$ coincides with the calss
$\FARL$.
\end{proposition}

\begin{proof}
Let $\mathscr{A}\in\mathbb K_{\mathrm{FGC}}$.  Conditions
\eqref{eq:KGC1} and \eqref{eq:KGC2} together state that
\[
 g(x)\le y\Longleftrightarrow x\le f(y),
\]
which is \eqref{FARL2}.  Conditions
\eqref{eq:KFGC1}--\eqref{eq:KFGC4} are
\eqref{FARL1} and \eqref{FARL3}--\eqref{FARL5}.  Hence
$\mathscr{A}\in\FARL$.  

The converse follows by reading
\eqref{FARL2} in its two directions and observing that the remaining the conditions of $\FARL$
are precisely the defining conditions of
$\mathbb K_{\mathrm{FGC}}$.
\end{proof}

\begin{theorem}\label{thm:FARL-strong-completeness}
For every $\Gamma\cup\{\varphi\}\subseteq
\mathrm{Fm}_{\mathcal L_{\mathrm{FGC}}}$,
\[
 \Gamma\vdash_{\mathbf{FL}^{\mathrm{FGC}}_{ew}}\varphi
 \Longleftrightarrow
 \Gamma\models_{\FARL}\varphi.
\]
Consequently, $\FARL$ is the equivalent algebraic semantics of
$\mathbf{FL}^{\mathrm{FGC}}_{ew}$.
\end{theorem}

\begin{proof}
It follows directly from Theorem~\ref{thm:FGC-completeness} and
Proposition~\ref{prop:identify-KFGC-FARL}.
\end{proof}

A formula is called \emph{FGC-free} if neither
$\triangledown$ nor $\blacktriangle$ occurs in it.

\begin{theorem}\label{thm:FGC-conservativity}
For every set $\Gamma$ of FGC-free formulas and FGC-free formula
$\varphi$,
\[
 \Gamma\vdash_{\mathbf{FL}^{\mathrm{FGC}}_{ew}}\varphi
 \Longleftrightarrow
 \Gamma\vdash_{\mathbf{FL}_{ew}}\varphi.
\]
\end{theorem}

\begin{proof}
The right-to-left implication holds because
$\mathbf{FL}^{\mathrm{FGC}}_{ew}$ extends $\mathbf{FL}_{ew}$.  Conversely,
suppose that $\Gamma\nvdash_{\mathbf{FL}_{ew}}\varphi$.  By strong completeness
of $\mathbf{FL}_{ew}$ with respect to residuated lattices, there are a
residuated lattice $L$ and a valuation $v$ such that
$v(\gamma)=1$ for all $\gamma\in\Gamma$, whereas $v(\varphi)\ne1$.
Expand $L$ by
\[
 f=\operatorname{id}_L,
 \qquad
 g=\operatorname{id}_L.
\]
\eqref{FARL1}--\eqref{FARL5} are immediate, so the expansion
belongs to $\FARL$.  Since all formulas are
FGC-free, their values are unchanged.  Therefore
$\Gamma\not\models_{\FARL}\varphi$, and
Theorem~\ref{thm:FARL-strong-completeness} gives
$\Gamma\nvdash_{\mathbf{FL}^{\mathrm{FGC}}_{ew}}\varphi$.
\end{proof}

Some basic algebraic properties of Frobenius-adjoint residuated lattices are listed. 
The proof of these properties is left as an exercise for the reader.

\begin{proposition}\label{prop:FARL-basic-properties}
Let $(\mathbf{L},f,g)\in\FARL$. For all $x,y\in L$, the following
properties hold:

\FARLtag{6}~~$x\leq f(g(x))$ and $g(f(x))\leq x$,

\FARLtag{7}~~$f(0)=0$,

\FARLtag{8}~~$x\leq g(x)$,

\FARLtag{9}~~$g(1)=1$,~$f(1)=1$,~$g(0)=0$,

\FARLtag{10}~~$f(g(x))=g(x)$,

\FARLtag{11}~~$x\leq y$ implies $f(x)\leq f(y)$ and
$g(x)\leq g(y)$,

\FARLtag{12}~~$g(f(x))=f(x)$,

\FARLtag{13}~~$f(f(x))=f(x)$ and $g(g(x))=g(x)$,

\FARLtag{14}~~$f(f(x)\vee y)=f(x)\vee f(y)$,

\FARLtag{15}~~$f(x\wedge y)=f(x)\wedge f(y)$ and
$g(x\vee y)=g(x)\vee g(y)$,

\FARLtag{16}~~$g(x\odot f(y))=g(x)\odot f(y)$,

\FARLtag{17}~~$f(f(x)\to y)=f(x)\to f(y)$,

\FARLtag{18}~~$f(x\to f(y))=g(x)\to f(y)$,

\FARLtag{19}~~$f(x\to g(y))=g(x)\to g(y)$,

\FARLtag{20}~~$f(x\to y)\leq f(x)\to f(y)$,

\FARLtag{21}~~$f(f(x)\to f(y))=f(x)\to f(y)$,

\FARLtag{22}~~$f(\neg f(x))=\neg f(x)$,

\FARLtag{23}~~$g(\neg g(x))=\neg g(x)$,

\FARLtag{24}~~$f(\neg x)=\neg g(x)$,

\FARLtag{25}~~$g(g(x)\odot g(y))=g(x)\odot g(y)$,

\FARLtag{26}~~$g(f(x)\odot f(y))=f(x)\odot f(y)$,

\FARLtag{27}~~$f(f(x)\odot f(y))=f(x)\odot f(y)$,

\FARLtag{28}~~$f(x)\leq y\Longleftrightarrow f(x)\leq f(y)$,

\FARLtag{29}~~$f(x\odot y)\geq f(x)\odot f(y)$.\\
If $L$ is involutive, then

\FARLtag{30}~~$\neg f(x)=g(\neg x)$.
\end{proposition}

The order-theoretic formulation of \eqref{FARL2} can be replaced by finitely many equations. This is
important both for algebraizability and for the finite model problem.

\begin{lemma}\label{lem:FARL-equational-adjunction} The class $\FARL$ is a finitely based variety. More precisely,
let $L$ be a bounded lattice and let $f,g:L\to L$.  Then the adjunction law
\[
 g(x)\le y\Longleftrightarrow x\le f(y)
\]
is equivalent to the following six conditions:

(A1)~ $x\le f(g(x))$,

 (A2)~$g(f(x))\le x$, 
 
 (A3)~$f(x\wedge y)=f(x)\wedge f(y)$, 
 
 (A4)~$f(1)=1$, 
 
 (A5)~$g(x\vee y)=g(x)\vee g(y)$, 
 
 (A6)~$g(0)=0$. 
\end{lemma}

\begin{proof}
If $g\dashv f$, then (A1) and (A2) are the unit and counit, while (A3)--(A6)
follow from the fact that right adjoints preserve finite meets and left adjoints
preserve finite joins.

Conversely, assume (A1)--(A6).  Conditions (A3) and (A5) imply that $f$ and
$g$ are monotone.  If $g(x)\le y$, then monotonicity of $f$ and (A1) give
$x\le f(g(x))\le f(y)$.
On the other hand, if $x\le f(y)$, then monotonicity of $g$ and (A2) give
 $g(x)\le g(f(y))\le y$. Thus $g(x)\le y$ if and only if $x\le f(y)$.

(A1), (A2) and \eqref{FARL1} are equations in the lattice
language, for example
$x\wedge f(g(x))=x$ and $g(f(x))\wedge x=g(f(x))$.  Hence
\eqref{FARL1}, \eqref{FARL3}--\eqref{FARL5}, together with the equational
forms of (A1)--(A6), constitute a finite equational basis for
$\FARL$, which is a finitely based variety.
\end{proof}

The definition of Frobenius-adjoint residuated lattices differs from that of monadic residuated lattices, which provide an algebraic semantics of monadic substructural predicate logic $\mathbf{mFL_{ew}\forall}$.

\begin{definition}\label{def:MRL}
Let
$\mathbf{L}$ be a residuated lattice and
$\forall,\exists:L\to L$ be two unary operators.  The algebra
\[(\mathbf{L},\forall,\exists)=(L,\wedge,\vee,\odot,\to,\forall,\exists,0,1)
\]
is called a \emph{monadic residuated lattice} if, for all
$x,y\in L$, it satisfies
\begin{align}
 \forall(x)\rightarrow x&\le1,
   \tag{Mon1}\label{Mon1}\\
 \forall(x\to \forall y)&=\exists x\to \forall y,
   \tag{Mon2}\label{Mon2}\\
 \forall(\forall x\to y)&=\forall x\to \forall y,
   \tag{Mon3}\label{Mon3}\\
 \forall(x\vee \exists y)&=\forall x\vee \exists y,
   \tag{Mon4}\label{Mon4}\\
\exists(x\odot x)&=\exists x\odot \exists x.
   \tag{Mon5}\label{Mon5}
\end{align}
The class of all monadic residuated lattices is denoted by
$\MRL$.
\end{definition}

\begin{example}
Let
\[
L=\{0,a,b,z,m,1\}
\]
be the lattice whose order is determined by
\[
0<a<z<m<1,
\qquad
0<b<z<m<1,
\]
with \(a\) and \(b\) incomparable, where
\[
a\wedge b=0,
\qquad
a\vee b=z.
\]
Define the monoidal product by
\[
x\odot y=
\begin{cases}
y, & x=1,\\
x, & y=1,\\
0, & x<1\ \text{and}\ y<1,
\end{cases}
\]
and define the residual operation by
\[
x\to y=
\begin{cases}
y, & x=1,\\
1, & x<1\ \text{and}\ x\leq y,\\
m, & x<1\ \text{and}\ x\nleq y.
\end{cases}
\]
A direct verification shows that
$\mathbf{L}=(L,\odot,\vee,\wedge,\to,0,1)$ is a residuated lattice. 

Define the unary operations \(f,g\colon L\to L\) by
\[
\begin{array}{c|cccccc}
x      &0&a&b&z&m&1\\
\midrule
\forall x   &0&a&0&a&m&1\\
\exists x   &0&a&m&m&m&1
\end{array}
\]
Then $(\mathbf{L},\forall,\exists) \in \MRL$, but $(\mathbf{L},\forall,\exists)\notin\FARL$, since
\eqref{FARL5} fails. In particular,

\[
\exists\bigl(b\wedge \exists a\bigr)
=
0
\neq
a
=
\exists b\wedge \exists a,
\]
which implies
\[
\MRL\nsubseteq \FARL.
\]
\end{example}

\begin{example}\label{exm:FARL2}
Consider the three-element \L{}ukasiewicz chain
\[
\L_{3}
=
\left\{0,\frac12,1\right\},
\]
with operations
\[
x\odot y=\max\{0,x+y-1\},
\qquad
x\to y=\min\{1,1-x+y\}.
\]
Define \(f,g\colon L_{3}\to L_{3}\) by
\[
f(x)=
\begin{cases}
1, & x=1,\\
0, & x<1,
\end{cases}
\qquad
g(x)=
\begin{cases}
0, & x=0,\\
1, & x>0.
\end{cases}
\]
Then $(\mathbf{L_{3}},f,g)\in\FARL$, but $(\mathbf{L_{3}},f,g)\notin\MRL$ since \eqref{Mon5} fails. In particular,
\[
g\left(\frac12\odot\frac12\right)
=
0\neq 1=g\left(\frac12\right)\odot g\left(\frac12\right),
\]
which implies
\[
\FARL\nsubseteq \MRL.
\]
\end{example}

\subsection{The finite model property of $\mathbf{FL}^{\mathrm{FGC}}_{ew}$}\label{subsection:FMP-FGC}

In this subsection, we study the finite model property of $\mathbf{FL}^{\mathrm{FGC}}_{ew}$, and characterize this property through finite $\FARL$-algebras and residual finiteness, establishing the strong finite model property for the FGC-free fragment.

\begin{definition}\label{def:FMP-FGC}
The logic $\mathbf{FL}^{\mathrm{FGC}}_{ew}$ has the \emph{finite model
property} (FMP) if, for every formula $\varphi$,
\[
 \nvdash_{\mathbf{FL}^{\mathrm{FGC}}_{ew}}\varphi
\]
implies that there are a finite algebra $\mathscr{L}\in\FARL$ and a
valuation $v$ such that $v(\varphi)\ne1$.

It has the \emph{strong finite model property} (SFMP) if, whenever $\Gamma$ is
finite and
\[
 \Gamma\nvdash_{\mathbf{FL}^{\mathrm{FGC}}_{ew}}\varphi,
\]
there are a finite $\mathscr{L}\in\FARL$ and a valuation $v$ such that
\[
 v(\gamma)=1\quad(\gamma\in\Gamma),
 \qquad
 v(\varphi)\ne1.
\]

We denote the class of finite members of $\FARL$ by
$\FARL_{\mathrm{fin}}$.
\end{definition}

Here we present an example w.r.t. the finite model property of the logic $\mathbf{FL}^{\mathrm{FGC}}_{ew}$.
\begin{example}\label{example:nontrivial-finite-countermodel}
Let
\[
\L_3=\left\{0,\frac{1}{2},1\right\}
\]
be the three-element \L ukasiewicz chain, and consider the
direct product
\[
L=\L_3\times\L_3,
\]
with all operations defined coordinatewise. Define
\(f,g:L\rightarrow L\) by
\[
f(x_1,x_2)
=
(x_1\wedge x_2,x_1\wedge x_2),~~~
g(x_1,x_2)
=
(x_1\vee x_2,x_1\vee x_2).
\]
Then
\[
\mathscr{L}=(\mathbf{L},f,g)\in\FARL.
\]

Define a valuation $v:\mathrm{Fm}_{\mathcal L_{\mathrm{FGC}}}\to\mathscr{L}$ by
\[
v(p)=(1,0),
\qquad
v(q)=(0,1).
\]
Then
\[
v(\blacktriangle p)
=
v(\blacktriangle q)
=
(1,1),
\]
whereas
\[
v\bigl(\blacktriangle(p\sqcap q)\bigr)
=
g(0,0)
=
(0,0).
\]
Therefore,
\[
v\left(
(\blacktriangle p\sqcap\blacktriangle q)
\Rightarrow
\blacktriangle(p\sqcap q)
\right)
=
(0,0)\neq(1,1).
\]
Thus $\mathscr{L}$ is a nontrivial finite countermodel to 
\[
(\blacktriangle p\sqcap\blacktriangle q)
\Rightarrow
\blacktriangle(p\sqcap q).
\]

Moreover, it is easy to verify that if
\[
\Gamma=\{\blacktriangle p,\blacktriangle q\}
\quad\text{and}\quad
\varphi=\blacktriangle(p\sqcap q),
\]
then
\[
v(\gamma)=1
\quad\text{for every }\gamma\in\Gamma,
\qquad
v(\varphi)\neq1.
\]
Hence this algebra also provides a countermodel
in the sense of the strong finite model property.
\end{example}

For a class $\mathcal K$ of similar algebras, $\mathbb V(\mathcal K)$ denotes
the variety generated by $\mathcal K$.  An algebra is \emph{residually finite}
if two distinct elements can be separated by a homomorphism into a finite
algebra.

\begin{theorem}\label{thm:FMP-characterization}
The following three conditions are equivalent:

  (1)~$\mathbf{FL}^{\mathrm{FGC}}_{ew}$ has the finite model property,
  
  (2)~$\FARL=\mathbb V(\FARL_{\mathrm{fin}})$,
  
  (3)~every finitely generated free $\FARL$-algebra is residually
 finite.

\end{theorem}

\begin{proof}
Assume (1), and let $s\approx t$ be an identity that fails in
$\FARL$.  Then there are $\mathscr{L}\in\FARL$ and a valuation
$v$ such that $v(s)\ne v(t)$.  In an integral residuated lattice,
\[
 v(s)=v(t)
 \Longleftrightarrow
 v(s\Leftrightarrow t)=1.
\]
Therefore $s\Leftrightarrow t$ is not valid in $\FARL$ and, by
Theorem~\ref{thm:FARL-strong-completeness}, is not a theorem of
$\mathbf{FL}^{\mathrm{FGC}}_{ew}$.  The FMP yields a finite
$\mathscr{B}\in\FARL$ and a valuation $w$ such that
$w(s\Leftrightarrow t)\ne1$, so $w(s)\ne w(t)$.  Thus every identity failing
in $\FARL$ already fails in a finite member.  By Birkhoff's theorem,
$\FARL=\mathbb V(\FARL_{\mathrm{fin}})$, proving (2).

Assume (2) and suppose
$\nvdash_{\mathbf{FL}^{\mathrm{FGC}}_{ew}}\varphi$.  By
Theorem~\ref{thm:FARL-strong-completeness}, the equation
$\varphi\approx1$ fails in $\FARL$.  Since $\FARL$ is
generated by its finite members, the same equation fails in some finite
$\mathscr{L}\in\FARL$.  Hence there is a valuation $v$ with
$v(\varphi)\ne1$, proving (1).

Assume (2), and let $\mathscr{F}_n$ be the free $\FARL$-algebra on
$n<\omega$ generators.  If $a,b\in F_n$ and $a\ne b$, choose terms $s,t$ in
$n$ variables representing $a,b$.  The identity $s\approx t$ fails in
$\mathscr{F}_n$, hence in $\FARL$.  By (2), it fails in a finite
$\mathscr{B}\in\FARL$ under some assignment of the variables.  By
freeness, that assignment extends to a homomorphism
$h:\mathscr{F}_n\to\mathscr{B}$ with $h(a)\ne h(b)$.  Thus $\mathscr{F}_n$ is
residually finite, proving (3).

Finally, assume (3), and let $s\approx t$ be an identity failing in
$\FARL$.  If $n$ variables occur in $s,t$, then the corresponding
elements $[s]$ and $[t]$ of $\mathscr{F}_n$ are distinct.  Residual finiteness
gives a homomorphism from $\mathscr{F}_n$ into a finite algebra
$\mathscr{B}\in\FARL$ that separates them.  Hence $s\approx t$ fails in
$\mathscr{B}$.  Thus $\FARL$ and
$\FARL_{\mathrm{fin}}$ have the same equational theory, and Birkhoff's
theorem yields (2).
\end{proof}

\begin{definition}\label{def:FEP-FARL}
A class $\mathcal K$ of algebras has the \emph{finite embeddability property}
(FEP) if every finite partial subalgebra of a member of $\mathcal K$ embeds, as
a partial algebra, into a finite member of $\mathcal K$.
\end{definition}

\begin{proposition}
\label{prop:FEP-implies-SFMP}
If $\FARL$ has the finite embeddability property, then
$\mathbf{FL}^{\mathrm{FGC}}_{ew}$ has the strong finite model property.
\end{proposition}

\begin{proof}
Let $\Gamma$ be finite and suppose that
$\Gamma\nvdash_{\mathbf{FL}^{\mathrm{FGC}}_{ew}}\varphi$.  By
Theorem~\ref{thm:FARL-strong-completeness}, there are
$\mathscr{L}\in\FARL$ and a valuation $v$ such that
\[
 v(\gamma)=1\quad(\gamma\in\Gamma),
 \qquad
 v(\varphi)\ne1.
\]
Let $\Sigma$ be the finite set of all subformulas of the formulas in
$\Gamma\cup\{\varphi\}$ and put
\[
 P=\{v(\psi):\psi\in\Sigma\}\cup\{0,1\}.
\]
Turn $P$ into a finite partial algebra of type
$\mathcal L_{\mathrm{FGC}}$ by defining a basic operation on a tuple from $P$
exactly when its value in $\mathscr{L}$ also belongs to $P$.  In particular, if
$\star(\psi_1,\ldots,\psi_k)\in\Sigma$, then the corresponding partial
operation is defined at
$(v(\psi_1),\ldots,v(\psi_k))$ and has value
$v(\star(\psi_1,\ldots,\psi_k))$.

By the FEP, there are a finite $\mathscr{B}\in\FARL$ and an injective
partial-algebra homomorphism $e:P\hookrightarrow B$.  Define a valuation $w$
in $\mathscr{B}$ by $w(p)=e(v(p))$ for every variable occurring in
$\Gamma\cup\{\varphi\}$, and arbitrarily on all remaining variables.  A
structural induction on $\psi\in\Sigma$ gives
\[
 w(\psi)=e(v(\psi)).
\]
Since $e$ preserves constants and is injective,
\[
 w(\gamma)=e(1)=1\quad(\gamma\in\Gamma),
 \qquad
 w(\varphi)=e(v(\varphi))\ne e(1)=1.
\]
Thus $\mathscr{B}$ is the required finite countermodel.
\end{proof}

The variety of residuated lattices has the finite embeddability property \cite{Galatos}, which gives a complete finite-model result for the
FGC-free fragment.

\begin{theorem}
\label{thm:FGC-free-SFMP}
Let $\Gamma$ be a finite set of FGC-free formulas and let $\varphi$ be
FGC-free.  If
\[
 \Gamma\nvdash_{\mathbf{FL}^{\mathrm{FGC}}_{ew}}\varphi,
\]
then there are a finite $\mathscr{L}\in\FARL$ and a valuation $v$ such
that
\[
 v(\gamma)=1\quad(\gamma\in\Gamma),
 \qquad
 v(\varphi)\ne1.
\]
Thus the FGC-free fragment has the strong finite model property.
\end{theorem}

\begin{proof}
Theorem~\ref{thm:FGC-conservativity} shows that
$\Gamma\nvdash_{\mathbf{FL}_{ew}}\varphi$.  The strong finite model property of
$\mathbf{FL}_{ew}$ yields a finite residuated lattice $L$ and a
valuation $v$ such that every member of $\Gamma$ has value $1$, whereas
$v(\varphi)\ne1$.  Expanding $L$ by
$f=g=\operatorname{id}_L$ produces a finite member of $\FARL$ and does
not change the values of FGC-free formulas.
\end{proof}

\section{A categorical equivalence of the categories $\FARDL$ and $\CFAD'$}\label{section4.}

In this section, we will extend the categorical equivalence $\mathbf{C} \dashv \mathbf{K}$. More precisely, we lift this equivalence $\mathbf{K}$ to the category $\mathbb{FARDL}$, whose objects are Frobenius-adjoint residuated distributive lattices, and the category $\mathbb{FAD\text{iff}RDL'}$, whose objects are pairs formed by an object of $\mathbb{D\text{iff}RDL'}$ and a center Frobenius-right adjoint operator.

\subsection{Involutive extension of Frobenius-adjoint residuated lattices}\label{subsection:INV-FARL}

In this subsection, we introduce the notion of Frobenius-adjoint involutive residuated lattices, and prove that it is a generalization of Frobenius-adjoint residuated lattices, which provides the algebraic foundation for establishing the equivalence between the categories $\FARDL$ and $\CFAD'$.

\begin{definition}\label{def:FAIRL}
Let
$\mathbf{L}$ be an involutive residuated lattice and
$f:L\to L$ a unary operator.  The algebra\\
\[
\mathscr{L}=(\mathbf{L},f)
  =(L,\wedge,\vee,\odot,\to,0,1,f)
\]
is called a \emph{Frobenius-adjoint involutive residuated lattice} if, for all
$x,y\in L$, it satisfies, 
\begin{align}
 f(x)&\leq x,
  \tag{FARL1}\label{FARL1}\\
 f(f(x)\vee y)&=f(x)\vee f(y),
   \tag{FARL14}\label{FARL14}\\
 f(f(x)\to y)&=f(x)\to f(y),
   \tag{FARL17}\label{FARL17}
\end{align}
The class of Frobenius-adjoint involutive residuated lattices is denoted by
$\FAIRL$, forms a variety.
\end{definition}

\begin{remark}\label{rem:FAIRL-basic-properties}
(1)~~Notice that Frobenius-adjoint involutive residuated lattice can be also denoted by $(\mathbf{L},f,g)$ with $g:=\neg f\neg$, we will henceforth simply write it as $\mathscr{L}$ for brevity.

(2)~~It is verified that \ref{FARL7}-\ref{FARL13} and \ref{FARL22}-\ref{FARL23} of Proposition \ref{prop:FARL-basic-properties} hold in any Frobenius-adjoint involutive residuated lattices, which are used to prove Theorem \ref{thm:FAIRL}. Conversely, the result of Theorem \ref{thm:FAIRL} shows that the conclusion of Proposition \ref{prop:FARL-basic-properties} also holds for Frobenius-adjoint involutive residuated lattices.
\end{remark}

\begin{theorem} \label{thm:FAIRL}
The involutive subvariety of the variety $\FARL$ determined by the
\begin{center} $\neg\neg x=x$
\end{center}
is term-equivalent to the variety $\FAIRL$.  
\end{theorem}
\begin{proof}
Let  $(\mathbf{L},f,g)\in \FARL$ such that $\mathbf{L}$ is an involutive residuated lattice. Then it follows from \eqref{FARL1} of Definition \ref{def:FARL}, \eqref{FARL14} and \eqref{FARL17} of Proposition \ref{prop:FARL-basic-properties} that  $(\mathbf{L},f)\in\FAIRL$.


Conversely, let $(\mathbf{L},f)\in\FAIRL$. Then we will show that $(\mathbf{L},f,g)\in \FARL$ with $g:=\neg f\neg$. 

(FARL1):\ It is directly follows from the condition \eqref{FARL1} of Definition \ref{def:FAIRL}.

(FARL2):\ If $g(x)\le y$, then by \ref{FARL8}, \ref{FARL10} and \ref{FARL11}, we have
\[
 x\le g(x)=f(g(x))\le f(y).
\]
Conversely, if $x\le f(y)$, then by \eqref{FARL1}, \ref{FARL11} and \ref{FARL12}, we have
\[
 g(x)\le g(f(y))= f(y)\leq y.
\]

(FARL3):\ For any $x,y\in L$, by \eqref{FARL3} and \eqref{FARL17}, we have
\[f(g(x)\rightarrow y)=f(f(g(x))\rightarrow y)= f(g(x))\rightarrow f(y)=g(x)\rightarrow f(y).    
\] 

(FARL4):\ For any $x,y\in L$, by \ref{FARL10} and \eqref{FARL14}, we have
\[f(g(x)\vee y)=f(f(g(x))\vee y)=f(g(x))\vee f(y)=g(x)\vee f(y).
\]

(FARL5):\ For any $x,y\in L$, by \eqref{FARL14}, we have
\[
 g(x\wedge g(y))=\neg f(\neg x\vee\neg g(y))=\neg f(\neg x\vee f(\neg y))=\neg\bigl(f(\neg x)\vee f(\neg y)\bigr)=g(x)\wedge g(y).
\]
\end{proof}

\begin{example}
\label{ex:block-nonMV-FAIRL}
Let
\[
N_4=\{0,a,b,1\},
\qquad
0<a<b<1,
\]
where \(a=\frac13\) and \(b=\frac23\). Equip \(N_4\) with the nilpotent
minimum operations
\[
x\odot y=
\begin{cases}
\min\{x,y\}, & x+y>1,\\
0,           & x+y\leq 1,
\end{cases}
\]
and
\[
x\to y=
\begin{cases}
1,                   & x\leq y,\\
\max\{1-x,y\},        & x>y.
\end{cases}
\]
The lattice operations are
\[
x\wedge y=\min\{x,y\},
\qquad
x\vee y=\max\{x,y\}.
\]
Then
\[
\mathbf N_4
=
\bigl(N_4,\wedge,\vee,\odot,\to,0,1\bigr)
\]
is a finite involutive residuated lattice.

Consider the direct power
\[
L=N_4^3,
\]
with all operations defined coordinatewise. Put
\[
\mathbf 0=(0,0,0),
\qquad
\mathbf 1=(1,1,1),
\]
and define \(f:L\to L\) by
\[
f(x_1,x_2,x_3)
=
(x_1\wedge x_2,x_1\wedge x_2,x_3).
\]
It is verified that
\[(\mathbf{L},f)
\]
is a Frobenius-adjoint involutive residuated lattice.
\end{example}

\begin{example} Let
\[
\mathbf{L}_{3}=\left\{0,\frac12,1\right\}
\]
be the three-element \L ukasiewicz chain, where
\begin{center}
$x\odot y=\max\{0,x+y-1\},~~
x\to y=\min\{1,1-x+y\}~\text{and}~
x\oplus y=\min\{1,x+y\}$.
\end{center}

Consider the direct product
\[
\mathbf{L}=\mathbf{L}_{3}\times \mathbf{L}_{3},
\]
with all operations defined coordinatewise, and
$\mathbf{0}=(0,0)~
\text{and}
~\mathbf{1}=(1,1)$.

Define the unary operations $f\colon L\to L$ by
\begin{center}
$f(x,y)=\bigl(\min\{x,y\},\min\{x,y\}\bigr).$
\end{center}
It is verified that  $(\mathbf{L},f)$ is a Frobenius-adjoint involutive residuated lattice.
\end{example} 

\subsection{Kalman structures derived from Frobenius-adjoint residuated distributive lattices}\label{subsection:KA-FARDL}

In this subsection, we introduce the notion of Frobenius-adjoint residuated lattices and prove that the classes of Frobenius-adjoint c-differential residuated distributive lattices satisfying $\mathbf{CK}$ and Frobenius-adjoint residuated distributive lattices are in one to one correspondence, which will play an important role in proving the main results of this section.

Let $(\mathbf{L},f,g)$ be a Frobenius-adjoint residuated lattice. Then we define a unary operator

\begin{center} $f_K:\mathbf{K}(L)\rightarrow \mathbf{K}(L)$
\end{center}
such that
\begin{center} $f_K(x,y)=(f(x), g(y))$,
\end{center}
for all $x,y\in \mathbf{K}(L)$.

\begin{theorem}\label{theorem4.1}
Let $(\mathbf{L},f,g)$ be a Frobenius-adjoint residuated distributive
lattice. Then, for any $(x,y),(x_1,y_1),(x_2,y_2)\in\mathbf{K}(L)$, the following properties hold:

\fKtag{1}\ $f_K$ is well defined, i.e.,
$f_K(x,y)\in\mathbf{K}(L)$,

\fKtag{2}\ $f_K(0,1)=(0,1)$,

\fKtag{3}\ $f_K(c)=c$,

\fKtag{4}\ $f_K(1,0)=(1,0)$,

\fKtag{5}\ $f_K(x,y)\leadsto(x,y)=(1,0)$,

\fKtag{6}\ 
$f_K\bigl((x_1,y_1)\leadsto(x_2,y_2)\bigr)
\leq
f_K(x_1,y_1)\leadsto f_K(x_2,y_2)$,

\fKtag{7}\ 
$f_K\bigl((x_1,y_1)\cap(x_2,y_2)\bigr)
=
f_K(x_1,y_1)\cap f_K(x_2,y_2)$,

\fKtag{8}\ 
$f_K\bigl(f_K(x_1,y_1)\leadsto(x_2,y_2)\bigr)
=
f_K(x_1,y_1)\leadsto f_K(x_2,y_2)$,

\fKtag{9}\ 
$f_K(x_1,y_1)\otimes f_K(x_2,y_2)
\leq
f_K\bigl((x_1,y_1)\otimes(x_2,y_2)\bigr)$,

\fKtag{10}\ 
$(x_1,y_1)\leq(x_2,y_2)$ implies
$f_K(x_1,y_1)\leq f_K(x_2,y_2)$,

\fKtag{11}\ 
$f_K\bigl(f_K(x_1,y_1)\leadsto f_K(x_2,y_2)\bigr)
=
f_K(x_1,y_1)\leadsto f_K(x_2,y_2)$,

\fKtag{12}\ 
$f_K\bigl(f_K(x_1,y_1)\cup(x_2,y_2)\bigr)
=
f_K(x_1,y_1)\cup f_K(x_2,y_2)$,

\fKtag{13}\ 
$f_K\bigl({\sim}f_K({\sim}(x,y))\bigr)
=
{\sim}f_K({\sim}(x,y))$,

\fKtag{14}\ 
$f_K(f_K(x,y))=f_K(x,y)$,

\fKtag{15}\ 
$f_K\bigl((x,y)\cap c\bigr)=f_K(x,y)\cap c$,

\fKtag{16}\ 
$f_K\bigl((x,y)\cup c\bigr)=f_K(x,y)\cup c$.
\end{theorem}

\begin{proof} $(f_K1)$:\ If $(x,y)\in \mathbf{K}(L)$, then $x\leq {\neg}y$. By (FRAL10) and (FRAL24) of Proposition \ref{prop:FARL-basic-properties}, 
\begin{center} $f(x)\leq f({\neg}y)={\neg}g (y)$,
\end{center} and hence $f(x)\odot g(y)=0$, that is $(f(x),g(y))\in \mathbf{K}(L)$, which implies that $f_K$ is well defined.

$(f_K2)$:\ From \ref{FARL7} and \ref{FARL9} of Proposition \ref{prop:FARL-basic-properties}, we have

\begin{center} $f_K(0,1)=(f(0),g(1))=(0,1)$.
\end{center}

$(f_K3)$:\ From \ref{FARL7} and \ref{FARL9} of Proposition \ref{prop:FARL-basic-properties}, we have

\begin{center} $f_K(c)=f_K(0,0)=(f(0),g(0))=(0,0)=c$.
\end{center}

$(f_K4)$:\ From \ref{FARL7} and \ref{FARL9} of Proposition \ref{prop:FARL-basic-properties}, we have 
\begin{center} $f_K(1,0)=(f(1),g(0))=(1,0)$.
\end{center}

$(f_K5)$:\ If $((x),y)\in \mathbf{K}(L)$, then by \ref{FARL8} and \ref{FARL11} of Proposition \ref{prop:FARL-basic-properties}, we have 
\begin{center}
$f (x)\odot y\leq f(x)\odot g(y)=0$,
\end{center}
and hence
\begin{center} $f(x)\odot y=0$, 
\end{center}
which implies
\begin{eqnarray*}
               f_K(x,y)\leadsto (x,y)&=& (f(x), g(y))\leadsto(x,y)\\&=& ((f(x)\rightarrow x)\wedge(y\rightarrow g(y)),f(x)\odot y)\\&=&(1,0).
                \end{eqnarray*}

$(f_K6)$:\ Let $(x_1,y_1),(x_2,y_2)\in \mathbf{K}(L)$. Then by \ref{FARL15}, \ref{FARL19} and \ref{FARL20} of Proposition \ref{prop:FARL-basic-properties}, we have

\begin{center} $f((x_1\rightarrow x_2)\wedge(y_2\rightarrow y_1))=f(x_1\rightarrow x_2)\wedge f (y_2\rightarrow y_1)\leq (f(x_1)\rightarrow f(x_2))\wedge (g(y_2)\rightarrow g(y_1)).$
\end{center}
Moreover, by \ref{FARL15}, \ref{FARL16} and \ref{FARL18} of Proposition \ref{prop:FARL-basic-properties}, we have
\begin{eqnarray*}
               f_K((x_1,y_1)\leadsto(x_2,y_2)) &=& f_K((x_1\rightarrow x_2)\wedge(y_2\rightarrow y_1), y_2\odot f(x_1))\\ &=& (f((x_1\rightarrow x_2)\wedge(y_2\rightarrow y_1)),g (y_2\odot f(x_1)))\\ &\leq & ((f(x_1)\rightarrow f(x_2))\wedge (g(y_2)\rightarrow g(y_1)), g(y_2)\odot f(x_1)).
               \end{eqnarray*}
On the other hand, we have
\begin{eqnarray*}
               f_K(x_1,y_1)\leadsto f_K(x_2,y_2)&=& (f(x_1),g(y_1))\leadsto (f(x_2),g(y_2))\\ &=& ((f(x_1)\rightarrow f (x_2))\wedge (g(y_2)\rightarrow g(y_1)), g (y_2)\odot f(x_1)).
               \end{eqnarray*}
                              
$(f_K7)$:\ Let $(x_1,y_1),(x_2,y_2)\in \mathbf{K}(L)$. Then by \ref{FARL18} of Definition \ref{def:FARL} and \eqref{FARL14} of Proposition \ref{prop:FARL-basic-properties}, we have
\begin{eqnarray*}
               f_K((x_1,y_1)\sqcap(x_2,y_2))&=& f_K(x_1\wedge x_2,y_1\vee y_2)\\ &=& (f(x_1\wedge x_2),g(y_1\vee y_2))\\ &=& (f (x_1)\wedge f(x_2),g(y_1)\vee g(y_2))\\ &=& (f(x_1),g(y_1))\sqcap (f(x_2), g(y_2))\\ &=& f_K(x_1,y_1)\sqcap f_K(x_2,y_2).
               \end{eqnarray*}

$(f_K8)$:\ Let $(x_1,y_1),(x_2,y_2)\in \mathbf{K}(L)$. Then by \ref{FARL15}--\eqref{FARL17} and \ref{FARL19} of Proposition \ref{prop:FARL-basic-properties}, we have
\begin{eqnarray*}
               f_K(f_K(x_1,y_1)\rightsquigarrow(x_2,y_2))&=& f_K((f(x_1),g(y_1))\rightsquigarrow(x_2,y_2))\\ &=& f_K((f (x_1)\rightarrow x_2)\wedge(y_2\rightarrow g(y_1)),f(x_1\odot y_2)\\ &=& (f((f (x_1)\rightarrow x_2)\wedge(y_2\rightarrow g(y_1))),g(f(x_1)\odot y_2))\\ &=& ((f(x_1)\rightarrow f(x_2))\wedge (g(y_2)\rightarrow g(y_1)),f(x_1)\odot g(y_2)).
               \end{eqnarray*}
On the other hand, we have
\begin{eqnarray*}
               f_K(x_1,y_1)\rightsquigarrow f_K(x_2,y_2)&=& (f(x_1),g(y_1))\rightsquigarrow (f(x_2),g(y_2))\\ &=& ((f (x_1)\rightarrow f(x_2))\wedge (g(y_2)\rightarrow g(y_1)),f(x_1)\odot g(y_2)).
               \end{eqnarray*}
               
$(f_K9)$:\  By \ref{fK6} and residuation, we have
\begin{center} $f_K((x_1,y_1)\leadsto (x_2,y_2))\otimes f_K(x_1,y_1)\leq f_K(x_2,y_2)$.
\end{center}
Taking $(x_2,y_2)=(x_2,y_2)\otimes (x_1,y_1)$ in the above equation, we have
\begin{eqnarray*} f_K(x_1,y_1)\otimes f_K(x_2,y_2)&=&f_K(x_2,y_2)\otimes f_K(x_1,y_1)\\&\leq&  f_K((x_1,y_1)\leadsto ((x_2,y_2)\otimes (x_1,y_1)))\otimes f_K(x_1,y_1)\\&\leq& f_K((x_1,y_1)\otimes (x_2,y_2)),
\end{eqnarray*} which implies
\begin{center} $f_K(x_1,y_1)\otimes f_K(x_2,y_2)\leq f_K((x_1,y_1)\otimes (x_2,y_2))$.
\end{center}

$(f_K10)$:\ If $(x_1,y_1)\leq (x_2,y_2)$, then by \ref{fK4} and \ref{fK6}, we have
\begin{center}$f_K((x_1,y_1)\leadsto (x_2,y_2))=f_K(1,0)=(1,0)\leq f_K(x_1,y_1)\leadsto f_K(x_2,y_2)$,
\end{center} which implies $f_K(x_1,y_1)\leq f_K(x_2,y_2)$.

$(f_K11)$:\ Let $(x_1,y_1),(x_2,y_2)\in \mathbf{K}(L)$. Then by \eqref{FARL16} and \eqref{FARL17} of Proposition \ref{prop:FARL-basic-properties}, we have
\begin{eqnarray*}
                f_K( f_K(x_1,y_1)\leadsto f_K(x_2,y_2))&=&  f_K((f(x_1),g(y_1))\leadsto (f(x_2),g (y_2))\\ &=& ((f(x_1)\rightarrow f(x_2))\wedge(g(y_2)\rightarrow g(y_1)),f(x_1)\odot g(y_2))\\ &=& (f(f(x_1)\rightarrow f (x_2))\wedge(g(y_2)\rightarrow g(y_1))),g (f(x_1)\odot g(y_2))\\ &=& ((f(x_1) \rightarrow f(x_2))\wedge (g (y_2)\rightarrow g(y_1)),f(x_1)\odot g(y_2)).
               \end{eqnarray*}
On the other hand, we have
\begin{eqnarray*}
                f_K(x_1,y_1)\leadsto  f_K(x_2,y_2)&=& (f(x_1),g(y_1))\leadsto (f(x_2),g(y_2))\\ &=& ((f(x_1) \rightarrow f (x_2))\wedge (g(y_2)\rightarrow g(y_1)),f(x_1)\odot g(y_2)),
               \end{eqnarray*} which implies 
\begin{center} $f_K( f_K(x_1,y_1)\leadsto f_K(x_2,y_2))=f_K(x_1,y_1)\leadsto  f_K(x_2,y_2)$.
\end{center}

$(f_K12)$:\ Let $(x_1,y_1),(x_2,y_2)\in \mathbf{K}(L)$. Then by \eqref{FARL5} of Definition \ref{def:FARL} and \eqref{FARL14} of Proposition \ref{prop:FARL-basic-properties}, we have
\begin{eqnarray*}
f_K(f_K(x_1,y_1)\cup (x_2,y_2))&=&f_K((f(x_1),g(y_1))\cup (x_2,y_2))\\&=&f_K(f(x_1)\vee x_2,g(y_1)\wedge y_2)\\&=&(f(f(x_1)\vee x_2),g(g(y_1)\wedge y_2))\\&=&(f(x_1)\vee f(x_2),g(y_1)\wedge g(y_2))\\&=&(f(x_1),g(y_1))\cup (f(x_2),g(y_2))\\&=&f_K(x_1,y_1)\cup f_K(x_2,y_2).
\end{eqnarray*}

$(f_K13)$:\ Let $(x,y)\in \mathbf{K}(L)$. Then by \ref{FARL10} and \ref{FARL13} of Proposition \ref{prop:FARL-basic-properties}, we have
\begin{eqnarray*}
               f_K({\sim}f_K({\sim}(x,y)))&=& f_K({\sim}f_K(y,x))\\&=& f_K(f_K(y,x)\leadsto(1,1))\\&=&f_K((f(y),g(x))\leadsto (1,1))\\&=&f_K(((f(y)\rightarrow 1)\wedge (1\rightarrow g(x)),g(x)))\\&=&f_K((g(x),g(x)))\\&=&(fg(x), g(x))\\&=&(g (x), g(x)).
               \end{eqnarray*}
On the other hand, we have
\begin{eqnarray*}
               {\sim}f_K({\sim}(x,y))&=& {\sim}f_K(y,x)\\&=& f_K(y,x)\leadsto(1,1)\\&=&(f(y),g(x))\leadsto (1,1)\\&=&((f (y)\rightarrow 1)\wedge (1\rightarrow g (x)),g (x))\\&=&(g (x),g (x)).
               \end{eqnarray*}
               
$(f_K14)$:\ Let $(x,y)\in \mathbf{K}(L)$. Then by \ref{FARL13} of Proposition \ref{prop:FARL-basic-properties}, we have
\begin{eqnarray*}
f_K(f_K(x,y))&=&f_K(f (x), g y)\\&=& (ff (x),g g y)\\&=&(f (x),g(y))\\&=& f_K(x,y).
\end{eqnarray*}

$(f_K15)$:\ Let $(x,y)\in \mathbf{K}(L)$. Then by \ref{FARL7} of Proposition \ref{prop:FARL-basic-properties}, we have
\begin{eqnarray*}
f_K((x,y)\cap c)&=&f_K((x,y)\cap(0,0))\\&=& f_K(0,y)\\&=&(f(0), g(y))\\&=& (0,g(y))\\&=&f_K(x,y)\cap c.
\end{eqnarray*}

$(f_K16)$:\ Let $(x,y)\in \mathbf{K}(L)$. Then by by \ref{FARL9} of Proposition \ref{prop:FARL-basic-properties}, we have
\begin{eqnarray*}
f_K((x,y)\cup c)&=&f_K((x,y)\cup(0,0))\\&=&f_K(x,0)\\&=&(f (x), g(0))\\&=&(f(x),0)\\&=&f_K(x,y)\cup c.
\end{eqnarray*}
\end{proof}

\begin{remark}\label{remark4.2} It is worth emphasizing that if $\mathbf{L}$ is a residuated distributive lattice, then $\mathbf{K}(L)$ is a special class of an involutive residuated lattice, with \ref{fK5},\ref{fK8} and \ref{fK12} of Theorem \ref{theorem4.1}, which forms a Frobenius-adjoint involutive residuated lattice in the sense of Definition \ref{def:FAIRL}.
\end{remark}

Motivated by Definition \ref{def:FAIRL},\ref{fK5},\ref{fK8}, \ref{fK12} of Theorem \ref{theorem4.1}, and the compatibility between Frobenius right adjoint operator and $c$ established in \ref{fK15} and \ref{fK16} of Theorem \ref{theorem4.1}, we introduce the notion of a center Frobenius right adjoint operator on c-differential residuated lattices, the resulting class of algebras will be called Frobenius adjoint c-differential residuated lattices.

\begin{definition}\label{def:centered-FA}
A \emph{Frobenius-adjoint c-differential residuated lattice} is an algebra
\[
 \mathcal{\mathscr{L}}_c=(\boldsymbol{\mathfrak{L}}, f_c)
 =(\mathcal{L},\cap,\cup,\otimes,\leadsto,\sim,f_c,0,c,1),
\]
where $\boldsymbol{\mathfrak{L}}=(\mathcal{L},\cap,\cup,\otimes,\leadsto,\sim,0,c,1)$ is a c-differential
residuated lattice and $f_c:\mathcal{L}\to \mathcal{L}$ satisfies
\begin{align*}
 f_c(c)&=c,\tag{$f_c1$}\label{fc1}\\
 f_c(x)\leadsto x&=1,\tag{$f_c2$}\label{fc2}\\
 f_c\bigl(f_c(x)\leadsto y\bigr)
   &=f_c(x)\leadsto f_c(y),\tag{$f_c3$}\label{fc3}\\
 f_c\bigl(f_c(x)\cup y)\bigr)
   &=f_c(x)\cup f_c(y),\tag{$f_c4$}\label{fc4}\\
 f_c(x\cap c)&=f_c(x)\cap c,\tag{$f_c5$}\label{fc5}\\
 f_c(x\cup c)&=f_c(x)\cup c,\tag{$f_c6$}\label{fc6}
\end{align*}
\end{definition}

The operator $f_c$ is called a \emph{center Frobenius right adjoint operator}.
We denote by $\CFAD$ the category whose objects are the distributive algebras
of this kind and whose morphisms preserve the full displayed signature.  The
full subcategory of $\CFAD$ consisting of the algebras satisfying $\mathbf{CK}$
is denoted by $\CFADp$.

\begin{remark}\label{rem:definition-explan.} (1) Taking $\boldsymbol{\mathfrak{L}}=\mathbf{K(L)}$ when $\mathbf{L}$ is a residuated distributive lattice, then by \ref{fK5},\ref{fK8},\\\ref{fK12},\ref{fK15} and \ref{fK16} of Theorem \ref{theorem4.1}, we have that $f_K$ is a center Frobenius right adjoint operator on the c-differential residuated lattice $\mathbf{K(L)}$, and it can be seen as the image of the Frobenius right and left adjoint operators $f$ and $g$, under the equivalence Kalman functor $\mathbf{K}$. 

(2) Note that c-differential residuated lattices are the special kinds of involutive residuated lattice, and hence the notion of Frobenius right adjoint operator is defined in Definition \ref{def:FAIRL}  can be also directly applied from involutive residuated lattices to c-differential residuated lattices. It is remarked here that Frobenius right adjoint operators and center Frobenius right adjoint operators on c-differential residuated lattices are not the same. For example, let $\mathbf{L}$ be a residuated distributive lattice and $\mathbf{K(L)}$ be the corresponding Kalman structure, 
\begin{center}$f_{\neg}:\mathbf{K}(L)\rightarrow \mathbf{K}(L)$ 
\end{center} be a unary operator defined by 
\begin{center} $f_{\neg}(x,y)=(x,\neg x)$,
\end{center} for any $(x,y)\in \mathbf{K}(L)$. Then $f_{\neg}$ is a Frobenius right adjoint operator of Definition \ref{def:FAIRL}, but it is not a center Frobenius right adjoint operator of Definition \ref{def:centered-FA} on the Kalman structure $\mathbf{K(L)}$.

 More precisely, if $(x,y)\in \mathbf{K}(L)$, then $x\odot \neg x=0$, which means that $f_{\neg}$ is well defined.

(\ref{FARL1}):\ If $(x,y)\in \mathbf{K}(L)$, then $ x\odot y=0$, and we have
\begin{eqnarray*}
               f_{\neg}(x,y)\leadsto (x,y)&=& (x,{\neg}x)\leadsto (x,y)\\&=&((x\rightarrow x)\wedge(y\rightarrow \neg x),x\odot y)\\&=&(1,0).
               \end{eqnarray*}

(\ref{FARL16}):\ Let $(x_1,y_1),(x_2,y_2)\in \mathbf{K}(L)$. Then
\[
\begin{aligned}
f_{\neg}\bigl(
    f_{\neg}(x_1,y_1)\cup(x_2,y_2)
\bigr)
&=
f_{\neg}\bigl(
    (x_1,\neg x_1)\cup(x_2,y_2)
\bigr)\\
&=
f_{\neg}\bigl(
    x_1\vee x_2,\,
    \neg x_1\wedge y_2
\bigr)\\
&=
\bigl(
    x_1\vee x_2,\,
    \neg(x_1\vee x_2)
\bigr)\\
&=
\bigl(
    x_1\vee x_2,\,
    \neg x_1\wedge\neg x_2
\bigr)\\
&=
(x_1,\neg x_1)\cup(x_2,\neg x_2)\\
&=
f_{\neg}(x_1,y_1)\cup f_{\neg}(x_2,y_2).
\end{aligned}
\]  
             
(\ref{FARL21}):\ Let $(x_1,y_1),(x_2,y_2)\in \mathbf{K}(L)$. First noticing that
 \begin{center} $x_1\odot (x_1\rightarrow x_2)\odot y_2\leq x_2\odot y_2=0$,
  \end{center}
 which directly implies $x_1\rightarrow x_2\leq y_2\rightarrow {-}x_1$ by (3) of Definition \ref{definietion2.1}. Then we have
\begin{eqnarray*}
               f_{\neg}(f_{\neg}(x_1,y_1)\leadsto(x_2,y_2))&=& f_{\neg}((x_1,\neg x_1)\leadsto (x_2,y_2))\\ &=& f_{\neg}((x_1\rightarrow x_2)\wedge(y_2\rightarrow \neg x_1),x_1\odot y_2)\\ &=& f_{\neg}(x_1\rightarrow x_2),x_1\odot y_2)\\ &=& (x_1\rightarrow x_2,\neg (x_1\rightarrow x_2))\\ &=& (x_1\rightarrow x_2,x_1\odot \neg x_2).
               \end{eqnarray*}
On the other hand, we have
\begin{eqnarray*}
               f_{\neg}(x_1,y_1)\leadsto f_{\neg}(x_2,y_2)&=& (x_1,\neg x_1)\leadsto (x_2,\neg x_2)\\ &=& ((x_1\rightarrow x_2)\wedge (\neg x_2\rightarrow \neg x_1), x_1\odot \neg x_2)\\ &=& (x_1\rightarrow x_2,x_1\odot \neg x_2).
               \end{eqnarray*}         
Then by Definition \ref{def:FAIRL} we have that $f_{\neg}$ is a  Frobenius right adjoint  on a c-differential residuated lattice $\mathbf{K(L)}$.
However, $f_{\neg}$ is not a center Frobenius right adjoint operator of Definition \ref{def:centered-FA} since Definition \ref{def:centered-FA} ($f_c1$) does not hold generally,
\begin{center} $f_{\neg}(c)=(0,0)=f_{\neg}(0,\neg 0)=(0,1)\neq (0,0)=c$.
\end{center}
\end{remark}

\begin{proposition}\label{proposition4.5}
In any Frobenius-adjoint $c$-differential residuated lattice
$\mathcal{\mathscr{L}}_c$, the following hold:

\fctag{7}\ $f_c(1)=1$,

\fctag{8}\ $f_c(0)=0$,

\fctag{9}\ $f_c(f_c(x))=f_c(x)$,

\fctag{10}\ $x\leq y$ implies $f_c(x)\leq f_c(y)$,

\fctag{11}\ $f_c(x)\otimes f_c(y)\leq f_c(x\otimes y)$,

\fctag{12}\ $f_c({\sim}f_c(x))={\sim}f_c(x)$,

\fctag{13}\ $f_c(x\cap y)=f_c(x)\cap f_c(y)$,

\fctag{14}\ $f_c(x\cup f_c(y))=f_c(x)\cup f_c(y)$,

\fctag{15}\ 
$f_c\bigl(f_c(x)\leadsto f_c(y)\bigr)
 =f_c(x)\leadsto f_c(y)$.
\end{proposition}

\begin{proof} The proofs of \ref{fc7}--\ref{fc15} can be deduced from Definition \ref{def:FARL} and Theorem \ref{theorem4.1}. 
\end{proof}

\begin{theorem}\label{theorem4.6} Let $\mathbf{L}$ be a residuated distributive lattice. Then
\[
\upsilon : (\mathbf L,f,g) \; \longmapsto \; (\mathbf{K(L)},f_K)
\]
defines a mapping from the class of  Frobenius adjoint expansions of $\mathbf L$ to the class of Frobenius adjoint expansions of $\mathbf{K(L)}$, which satisfies the following properties:

   (1)~~If $(\mathbf{L}, f_1, g_1)$ and $(\mathbf{L}, f_2, g_2)$ are two Frobenius adjoint expansions of $\mathbf{L}$, then
    \[
    \mathbf{K}(\mathbf{L}, f_1, g_1) \cong \mathbf{K}(\mathbf{L}, f_2, g_2)
    \]
    as Frobenius adjoint expansions of $\mathbf{K}(\mathbf{L})$, which shows that the image under $\mathbf{K}$ of any Frobenius adjoint expansions of $\mathbf{L}$ is essentially unique.

   (2)~~If $(\mathbf{K(L)}, f_K)$ is a Frobenius adjoint c-differential residuated distributive lattice, then there exist unary operators $f,g$ such that $(\mathbf{L}, f, g)$ is a Frobenius adjoint residuated distributive lattice and
    \[
   \upsilon(\mathbf{L}, f, g) \cong (\mathbf{K(L)}, f_K).
    \]

Consequently, $\upsilon$ establishes a one-to-one correspondence (up to isomorphism) between the class of Frobenius adjoint expansions of $\mathbf{L}$ to the class of Frobenius adjoint expansions of $\mathbf{K(L)}$.
\end{theorem}
\begin{proof}
(1)~~If $(\mathbf{L},f,g)$ is a Frobenius adjoint residuated distributive lattice, then by Theorem \ref{theorem4.1}\ref{fK5},\\\ref{fK8},\ref{fK12},\ref{fK15} and \ref{fK16} we have that  $(\mathbf{K(L)},f_K)$ satisfies the conditions \eqref{fc1}--\eqref{fc6}, which shows that $(\mathbf{K(L)},f_K)$ is a Frobenius adjoint c-differential residuated distributive lattice. 

Let $(\mathbf{L},f_1,f_1)$ and $(\mathbf{L},f_2,f_2)$ be two Frobenius adjoint residuated distributive lattices such that
\begin{center}$f_{K_1}(x,y)=(f_1(x),g_1(y))=(f_2(x),g_2(y))=f_{K_2}(x,y)$
\end{center}
for any $(x,y)\in \mathbf{K}(L)$. In particular, as $(x,{\neg}x)\in \mathbf{K}(L)$, we have
\begin{center}
$(f_1(x),g_1(\neg x))=(f_2(x),g_2(\neg x))$,
\end{center}
By \ref{FARL24} of Proposition \ref{prop:FARL-basic-properties},
Frobenius left adjoints are uniquely determined by Frobenius right adjoints. Hence, since 
$f_1=f_2$, we conclude that $g_1=g_2$.


  (2)~~If $(\mathbf{K(L)},f_K)$ is a Frobenius adjoint c-difference residuated distributive lattice (by Theorem \ref{theorem4.1} this is in fact the case), then we define the function $f:L\rightarrow L$ by
\begin{center} $f(x)=\pi_1f_K(x,0)$,
\end{center}
proving that $(\mathbf{L},f,g)$ is a Frobenius adjoint residuated distributive lattice, in which $g(x):=\neg f(\neg x)$. Here we first prove that $(\mathbf{L},f)$ is a Frobenius adjoint involutive residuated distributive lattice. By Proposition \ref{proposition4.5}, we have
\begin{center}~~~$(\star)$~~~ $f_K(x,0)=f_K((x,0)\cup c)=(f(x),0)$.
\end{center}

(FARL1):\ By \eqref{fc2} of Definition \ref{def:centered-FA} and $(\star)$, we have
\begin{center} $(1,0)=f_K(x,0)\leadsto(x,0)=(f(x),0)\leadsto (x,0)=(f(x)\rightarrow x,0)$,
\end{center}
which implies
\begin{center} $f(x)\rightarrow x=1$.
\end{center}

(FARL16):\ By $(\star)$ and \eqref{fc14} of Proposition \ref{proposition4.5}, we have
\[
\begin{aligned}
\bigl(f(f(x)\vee y),0\bigr)
&=
f_K\bigl(f(x)\vee y,0\bigr)\\
&=
f_K\bigl((f(x),0)\cup(y,0)\bigr)\\
&=
f_K\bigl(f_K(x,0)\cup(y,0)\bigr)\\
&=
f_K(x,0)\cup f_K(y,0)\\
&=
(f(x),0)\cup(f(y),0)\\
&=
\bigl(f(x)\vee f(y),0\bigr),
\end{aligned}
\]
which gives
\[
f\bigl(f(x)\vee y\bigr)=f(x)\vee f(y).
\]

(FARL21):\ By $(\star)$, we have
\begin{center} $f(x\leadsto y,0)=(f(x),0)\leadsto (y,0)=f_K(x,0)\leadsto (y,0)$,
\end{center}
and hence
\begin{eqnarray*}
               (f(f(x)\rightarrow y),0)&=&f_K(f_K(x,0)\leadsto (y,0))\\&=&f_K(x,0)\leadsto f_K(y,0)\\&=&(f(x),0)\rightarrow (f(y),0)\\&=&(f(x)\rightarrow f(y),0),
               \end{eqnarray*}
which proves
\begin{center} $f(f(x)\rightarrow y)=f(x)\rightarrow f(y)$.
 \end{center}
Hence by Theorem \ref{thm:FAIRL} that $(\mathbf{L},f,g)$ is a Frobenius adjoint residuated distributive lattice.

Moreover, we prove that $f_K=f_c$. By Proposition \ref{proposition4.5}, we have
\begin{center} $f_K(x,0)=(f(x),0)=f_c(x,0)$,
\end{center}
and
\begin{eqnarray*}
              f_c(0,y)&=&f_c(({\neg}y,0)\otimes c)\\&=& f_c({\neg}y,0)\otimes c\\&=&(f {\neg}y,0)\otimes c\\&=& (0,f(y))\\&=& f_K(0,y).
               \end{eqnarray*}
Here we also show that
\begin{center} $f_c(x,y)\cup c=f_c(x,0)=(f(x),0)=f_K(x,y)\cup c$,

$f_c(x,y)\cap c=f_c(x,0)=(0,f(y))=f_K(x,y)\cap c$.
\end{center}
Therefore the result follows by distributivity. More precisely, we have 
\begin{eqnarray*}
              f_c(x,y)&=&f_c(x,y)\cap(f_c(x,y)\cup c)\\&=&f_c(x,y)\cap(f_K(x,y)\cup c)\\&=&(f_c(x,y)\cap f_K(x,y))\cup (f_c(x,y)\cap c)\\&=&(f_K(x,y)\cap f_c(x,y))\cup (f_K(x,y)\cap c)\\&=&f_K(x,y)\cap(f_c(x,y)\cup c)\\&=&f_K(x,y)\cap(f_K(x,y)\cup c)\\&=&f_K(x,y),
               \end{eqnarray*} which implies $f_K=f_c$.
\end{proof}

 The following remark highlights the relevance of Theorem \ref{theorem4.6} through an application involving Example \ref{exm:FARL2}.

\begin{remark}\label{example5.9.} Let $(\mathbf{L},f,g)$ be the Frobenius adjoint residuated distributive lattice  in Example \ref{exm:FARL2}. 
Here it is checked that $(\mathbf{K}(L),\cap,\cup,\otimes,\leadsto,(1,0),(0,0),(0,1))$ is a c-differential residuated distributive lattice with $c=(0,0)$, where 
\begin{center} $\mathbf{K}(L)=\{(0,1),(0,\frac{1}{2}),(0,0),(\frac{1}{2},\frac{1}{2}),(\frac{1}{2},0),(1,0)\}$,
\end{center}
 and the Hasse diagram as follows:
\begin{center}
  \begin{tikzpicture}
  \node[circle, fill, inner sep=2pt] (a) at (1,1) {};
  \node[circle, fill, inner sep=2pt] (b) at (-1,1) {};
  \node[circle, fill, inner sep=2pt] (c) at (0,2) {};
   \node[circle, fill, inner sep=2pt] (d) at (0,3) {};
  \node[circle, fill, inner sep=2pt] (zero) at (0,0) {};
   \node[circle, fill, inner sep=2pt] (e) at (0,-1) {};
  \node at (a) [right=3pt] {$(0,0)$};
  \node at (b) [left=3pt] {$(\frac{1}{2},\frac{1}{2})$};
  \node at (c) [left=3pt] {$(\frac{1}{2},0)$};
  \node at (d) [left=3pt] {$(1,0)$};
   \node at (e) [left=3pt] {$(0,1)$};
  \node at (zero) [right=3pt] {$(0,\frac{1}{2})$};
  \draw (e)--(zero) -- (a) -- (c) --(d)--(c)--(b) -- (zero)--(e);
  \draw (c) -- (a);
\end{tikzpicture}
\end{center}
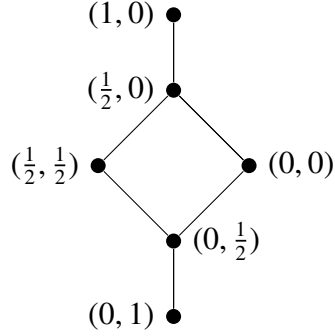
\begin{center} {\bf{Figure 8.}} Hasse diagram of Remark \ref{example5.9.}
\end{center}
Then it follows from Theorem \ref{theorem4.6} that the algebraic structure $$(\mathbf{K}(L),\cap,\cup,\otimes,\leadsto,f_K,(1,0),(0,0),(0,1))$$ is a Frobenius adjoint c-differential residuated distributive lattice, where
\begin{center}
$f_K(x,y)=
\begin{cases}
(1,0), & (x,y)= (1,0),\\
(0,0), & (x,y)= (0,0),(\frac{1}{2},0),\\
(0,1), & (x,y)= (\frac{1}{2},\frac{1}{2}),(0,\frac{1}{2}),(0,1),
\end{cases}$
\end{center}
and its corresponding dual operator $g_K$ as follows:
\begin{center}
$g_K(x,y)=
\begin{cases}
(1,0), & (x,y)= (1,0),(\frac{1}{2},0),(\frac{1}{2},\frac{1}{2}),\\
(0,0), & (x,y)= (0,0),(0,\frac{1}{2}),\\
(0,1), & (x,y)= (0,1).
\end{cases}$
\end{center}

Conversely, if the Frobenius adjoint $c$-differential residuated distributive lattice $$(\mathbf{K}(L),\cap,\cup,\otimes,\leadsto,f_K,g_K,(1,0),(0,0),(0,1))$$ is given in advance as above, then by Theorem \ref{theorem4.6}, one can directly recover the Frobenius adjoint residuated distributive lattice $(\mathbf{L},f,g)$ constructed by Example \ref{exm:FARL2}.
\end{remark}

\subsection{Extending the categorical equivalence between the categories $\mathbb{RDL}$ and $\mathbb{D\textmd{iff}RDL'}$ }\label{subsection:CAE-FARL}

In this subsection, we will extend the categorical equivalence $\mathbf{C} \dashv \mathbf{K}$. Specifically, here we define two extended functors, which we also call $\mathbf{C}$ and $\mathbf{K}$, between the category $\mathbb{FARDL}$, whose objects are Frobenius adjoint residuated distributive lattices, and the category $\mathbb{FAD\text{iff}RDL}$, whose objects are Frobenius adjoint c-differential residuated distributive lattices. In both categories, the morphisms are the corresponding algebra homomorphisms.

It is easy to check that the map
\begin{center}
$\mathbf{K}_{\phi}:(\mathbf{K(L)}_1,f_{K_1},g_{K_1})\rightarrow (\mathbf{K(L)}_2,f_{K_2},g_{K_2})$
 \end{center} given by
 \begin{center} $\mathbf{K}_{\phi}(x,y)=(\phi(x),\phi(y))$
 \end{center} is indeed a morphism in $\mathbb{FAD\text{iff}RDL}$ from $(\mathbf{K(L)}_1,f_{K_1},g_{K_1})$ to $(\mathbf{K(L)}_2,f_{K_2},g_{K_2})$. These assignments establish a functor $\mathbf{K}$ from $\mathbb{FARDL}$ to $\mathbb{FAD\text{iff}RDL}$.

Moreover, if $\epsilon:\boldsymbol{\mathfrak{L}}_1\rightarrow \boldsymbol{\mathfrak{L}}_2$ is a homomorphism of c-differential residuated distributive lattices, then it is not hard to see that
\begin{center}$\mathbf{C}_{\epsilon}:\boldsymbol{\mathfrak{L}}_1\rightarrow \boldsymbol{\mathfrak{L}}_2$
\end{center} defined by
\begin{center} $\mathbf{C}_\epsilon(x)=\epsilon(x)$
\end{center} is a homomorphism of residuated distributive lattices.

The above construction can be lifted to the corresponding Frobenius adjoint algebras.

\begin{proposition}\label{proposition5.7.} Let $\boldsymbol{(\mathfrak{L}},f_c)\in \mathbb{FAD\text{iff}RDL}$. Then $(\mathcal{C\boldsymbol{(\mathfrak{L}})},f_c,g_c)\in\mathbb{FARDL}$, where
\begin{center} $g_c(x):=\sim f_c(\sim x)$ for any $x\in \mathcal{L}$.
\end{center}
If
$\epsilon:\boldsymbol{\mathfrak{L}}_1\rightarrow \boldsymbol{\mathfrak{L}}_2$  is a morphism in $\mathbb{FADRDL}$, then
$\mathbf{C}_\epsilon:\mathcal{C\boldsymbol{(\mathfrak{L_\text{1}}})}\rightarrow \mathcal{C\boldsymbol{(\mathfrak{L_\text{2}}})}$ is a morphism in $\mathbb{FARDL}$.
\end{proposition}
\begin{proof} We show that $f_c$ and $g_c$ are well defined in $\boldsymbol{\mathfrak{L}}$, and then prove that 
$(\boldsymbol{\mathfrak{L}},f_c,g_c)\in\mathbb{FARDL}$.  

In particular, if $x\in \mathcal{C(L)}$, then $x=x\cup c$. By \eqref{fc1} and \ref{fc15}, we have
\begin{center} $f_c(x\cup c)=f_c(x\cup f_c c)=f_c (x)\cup f_c(c)$, 
\end{center} which implies
\begin{center}$f_c(x)=f(x)\cup c$,
\end{center} and hence $f_c(x)\in\mathcal{C(L)}$. And if $x\in \mathcal{C(L)}$, i.e., $x\geq c$, then ${\sim} x\leq {\sim} c$. By  \eqref{fc2} and \ref{fc10}, we have
\begin{center} $f_c(\sim x)\leq f_c(\sim c)=\sim c$,
\end{center} which implies
\begin{center} $g_c(x)=\sim f_c(\sim x)\geq {\sim}({\sim} c)=c$.
\end{center} Hence $g_c(x)\geq c$, which shows $g_c(x)\in \mathcal{C(L)}$.

Moreover, \eqref{FARL1} can be obtained from \eqref{fc1}.
The proofs of \eqref{FARL2}--\eqref{FARL5}  are similar to those of \eqref{FARL2}--\eqref{FARL5} in the derivation from Theorem \ref{thm:FAIRL}, respectively. The other part can be obtained from the fact that $\mathbf{C}_\epsilon$ is a homomorphism of residuated distributive lattices.
\end{proof}

\begin{proposition}\label{proposition5.12.} Let $(\mathbf{L},f,g)$ be a Frobenius adjoint residuated distributive lattice. Then the map
\begin{center} $\rho:(\mathbf{L},f,g)\rightarrow (\mathcal{C(\mathbf{K(L)}},f_K,g_K)$
\end{center} given by
\begin{center} $\rho(x)=(x,0)$
\end{center} is an isomorphism in $\mathbb{FARDL}$.
\end{proposition}
\begin{proof} 
It is well known (see \cite{Castiglionib}) that $\rho$ is a bijective homomorphism of distributive residuated lattices. To prove our claim, it remains to show that it preserves both $f$ and $g$. Indeed, let $x \in L$. It follows from \ref{FARL7} and \ref{FARL9} of Proposition \ref{prop:FARL-basic-properties} that
\begin{center} $\rho(f(x))=(f(x),0)=(f(x),g(0))=f_K((x,0))=f_K(\rho(x))$,

$\rho(g(x))=(g(x),0)=(g(x), g(0))=g_K((x,0))=g_K(\rho(x))$,

\end{center} which proves that $\rho$ preserves the Frobenius right and left adjoint operators.
\end{proof}

If $\boldsymbol{\mathfrak{L}}$ is a c-differential residuated distributive lattice, then
\begin{center} $\sigma:\boldsymbol{\mathfrak{L}}\rightarrow  \mathbf{K}(\mathcal{C(\boldsymbol{\mathfrak{L}})})$
\end{center} given by
\begin{center} $\sigma(x)=(x\cup c,\sim x\cup c)$
\end{center} is an injective homomorphism of c-differential residuated distributive lattices (see \cite{Castiglionib}).

To obtain the main result of this subsection, we need the following important results.

\begin{proposition}\label{proposition5.13} Let $(\boldsymbol{\mathfrak{L}},f_c)$ be a Frobenius adjoint c-differential residuated distributive lattice. Then the injective map $\sigma$ preserves $f_c$.
\end{proposition}
\begin{proof}
First, note that $g_c(c) = c$. Let $x \in \mathcal{L}$. Then from Definition \ref{def:centered-FA}, the definition of $g_c$ and \ref{FARL15} of Proposition \ref{prop:FARL-basic-properties}, we have:
\begin{eqnarray*}
\sigma(f_c(x)) &=& (f_c(x) \cup c, \sim f_c (x) \cup c) \\
&=& (f_c (x \cup c), g_c(\sim x) \cup c) \\
&=& (f_c(x \cup c), g_c(\sim x \cup c)) \\
&=& f_K(x \cup c, \sim x \cup c) \\
&=& f_K(\sigma(x)).
\end{eqnarray*}
Thus,
\begin{center}
$\sigma(f_c(x)) = f_K(\sigma(x))$.
\end{center}
This shows that the map $\sigma$ preserves the center Frobenius right adjoint operator.
\end{proof}

\begin{proposition}\label{proposition5.14}\cite{Castiglionib}  A c-differential residuated distributive lattice $\boldsymbol{\mathfrak{L}}$ satisfies $\mathbf{CK}$ if and only if $\sigma$ is surjective.
\end{proposition}

Here we denote by $\mathbb{FAD\text{iff}RDL'}$ the full subcategory of $\mathbb{FAD\text{iff}RDL}$ whose objects satisfy $\mathbf{CK}$. As a consequence of the previous results, we are ready to present the main result of the paper.
\begin{theorem}\label{theorem5.15}
The functors
\[
\mathbf{K}:\mathbb{FARDL}\longrightarrow\mathbb{FAD\text{iff}RDL'}
\qquad\text{and}\qquad
\mathbf{C}:\mathbb{FAD\text{iff}RDL'}\longrightarrow\mathbb{FARDL}
\]
are quasi-inverse. More precisely, the families $\rho$ and $\sigma$ defined above form natural isomorphisms
\[
\rho:\operatorname{Id}_{\mathbb{FARDL}}
   \Longrightarrow \mathbf{C}\mathbf{K}
\qquad\text{and}\qquad
\sigma:\operatorname{Id}_{\mathbb{FAD\text{iff}RDL'}}
   \Longrightarrow \mathbf{K}\mathbf{C}.
\]
Consequently,
\[
\mathbb{FARDL}\simeq\mathbb{FAD\text{iff}RDL'}.
\]
\end{theorem}

\begin{proof}
First, we verify that the two constructions considered above restrict to the stated categories. The coordinatewise assignment $\mathbf{K}$ defines a functor from $\mathbb{FARDL}$ to $\mathbb{FAD\text{iff}RDL}$. Since the underlying Kalman algebra $\mathbf{K}(\mathbf{L})$ belongs to $\mathbb{D\text{iff}RDL'}$ for every residuated distributive lattice $\mathbf{L}$, the expanded algebra $(\mathbf{K}(\mathbf{L}),f_K)$ satisfies $\mathbf{CK}$ as well. Hence $\mathbf{K}$ restricts to a functor
\[
\mathbf{K}:\mathbb{FARDL}\longrightarrow\mathbb{FAD\text{iff}RDL'}.
\]
Proposition~\ref{proposition5.7.} also shows that the positive-cone construction gives a functor
\[
\mathbf{C}:\mathbb{FAD\text{iff}RDL'}\longrightarrow\mathbb{FARDL}.
\]

Let $(\mathbf{L},f,g)\in\mathbb{FARDL}$. Then by Proposition~\ref{proposition5.12.}, the map
\[
\rho_{\mathbf{L}}:(\mathbf{L},f,g)
   \longrightarrow \mathbf{C}\mathbf{K}(\mathbf{L},f,g),
\qquad
\rho_{\mathbf{L}}(x)=(x,0),
\]
is an isomorphism in $\mathbb{FARDL}$. We claim that the family
$\rho=\{\rho_{\mathbf{L}}\}$ is natural. Let
\[
\phi:(\mathbf{L}_1,f_1,g_1)\longrightarrow(\mathbf{L}_2,f_2,g_2)
\]
be a morphism in $\mathbb{FARDL}$. Since $\mathbf{K}_{\phi}$ acts coordinatewise and $\phi(0)=0$, for every $x\in L_1$ we have
\[
\begin{aligned}
(\mathbf{C}\mathbf{K}\phi)\bigl(\rho_{\mathbf{L}_1}(x)\bigr)
  &=(\mathbf{C}\mathbf{K}\phi)(x,0)\\
  &=(\phi(x),\phi(0))\\
  &=(\phi(x),0)\\
  &=\rho_{\mathbf{L}_2}(\phi(x)).
\end{aligned}
\]
Therefore,
\[
\mathbf{C}\mathbf{K}\phi\circ\rho_{\mathbf{L}_1}
   =\rho_{\mathbf{L}_2}\circ\phi,
\]
so that
\[
\rho:\operatorname{Id}_{\mathbb{FARDL}}
   \Longrightarrow\mathbf{C}\mathbf{K}
\]
is a natural isomorphism.

Now let $(\boldsymbol{\mathfrak{L}},f_c)\in\mathbb{FAD\text{iff}RDL'}$. Recall that
\[
\sigma_{\boldsymbol{\mathfrak{L}}}:\boldsymbol{\mathfrak{L}}
   \longrightarrow\mathbf{K}\mathbf{C}(\boldsymbol{\mathfrak{L}}),
\qquad
\sigma_{\boldsymbol{\mathfrak{L}}}(x)
   =(x\cup c,\sim x\cup c),
\]
is an injective homomorphism of $c$-differential residuated distributive lattices. By Proposition~\ref{proposition5.13}, it also preserves the center Frobenius right adjoint operator $f_c$. Since $\boldsymbol{\mathfrak{L}}$ satisfies $\mathbf{CK}$, Proposition~\ref{proposition5.14} yields the surjectivity of $\sigma_{\boldsymbol{\mathfrak{L}}}$. Hence $\sigma_{\boldsymbol{\mathfrak{L}}}$ is an isomorphism in $\mathbb{FAD\text{iff}RDL'}$.

Here it remains to prove naturality. Let
\[
\epsilon:(\boldsymbol{\mathfrak{L}}_1,f_{c,1})
  \longrightarrow(\boldsymbol{\mathfrak{L}}_2,f_{c,2})
\]
be a morphism in $\mathbb{FAD\text{iff}RDL'}$. Since $\epsilon$ preserves $\cup$, $\sim$, and the distinguished center $c$, while $\mathbf{K}\mathbf{C}\epsilon$ acts coordinatewise, for every $x\in\mathcal{L}_1$ we obtain
\[
\begin{aligned}
(\mathbf{K}\mathbf{C}\epsilon)
   \bigl(\sigma_{\boldsymbol{\mathfrak{L}}_1}(x)\bigr)
  &=(\mathbf{K}\mathbf{C}\epsilon)(x\cup c,\sim x\cup c)\\
  &=\bigl(\epsilon(x\cup c),\epsilon(\sim x\cup c)\bigr)\\
  &=\bigl(\epsilon(x)\cup c,\sim\epsilon(x)\cup c\bigr)\\
  &=\sigma_{\boldsymbol{\mathfrak{L}}_2}(\epsilon(x)).
\end{aligned}
\]
Thus,
\[
\mathbf{K}\mathbf{C}\epsilon\circ\sigma_{\boldsymbol{\mathfrak{L}}_1}
   =\sigma_{\boldsymbol{\mathfrak{L}}_2}\circ\epsilon,
\]
which implies that
\[
\sigma:\operatorname{Id}_{\mathbb{FAD\text{iff}RDL'}}
   \Longrightarrow\mathbf{K}\mathbf{C}
\]
is a natural isomorphism.

We have therefore obtained natural isomorphisms
\[
\operatorname{Id}_{\mathbb{FARDL}}\cong\mathbf{C}\mathbf{K}
\qquad\text{and}\qquad
\operatorname{Id}_{\mathbb{FAD\text{iff}RDL'}}\cong\mathbf{K}\mathbf{C}.
\]
Therefore $\mathbf{K}$ and $\mathbf{C}$ are quasi-inverse functors and establish the categorical equivalence.
\end{proof}

\begin{corollary}\label{corollary5.16}
    The functor $\mathbb{FARDL}\vdash \mathbb{RDL}$ followed by $\mathbf{K}:\mathbb{RDL}\vdash \mathbb{D\text{iff}RDL'}$ factors through the forgetful $\mathbf{F}:\mathbb {FAD\text{iff}RDL'}\vdash  \mathbb{D\text{iff}RDL'}$.
\end{corollary}
\begin{proof} The proof is straightforward.
\end{proof}

As a consequence of Theorem \ref{theorem5.15} and Corollary \ref{corollary5.16}, we obtain a commutative diagram:
\begin{center}
\begin{tikzpicture}
  \node[inner sep=2pt] (MV^bullet) at (0,0) {$\mathbb{RDL}$};
  \node[inner sep=2pt] (MV) at (0,3) {$\mathbb{FARDL}$ };
  \node[inner sep=2pt] (MDRL) at (3,0) {$\mathbb{D\text{iff}RDL'}$};
  \node[inner sep=2pt] (iIRL_0) at (3,3) {$\mathbb{FAD\text{iff}RDL'}$};  
\draw[->] ([yshift=3pt]MV.east) -- ([yshift=3pt]iIRL_0.west) node[midway, above=3pt] {$\mathbf{K}$};
\draw[<-] ([yshift=-3pt]MV.east) -- ([yshift=-3pt]iIRL_0.west) node[midway, below=3pt] {$\mathbf{C}$};
\draw[->] ([yshift=3pt]MDRL.west) -- ([yshift=3pt]MV^bullet.east) node[midway, above=3pt] {$\mathbf{K}$};
\draw[<-] ([yshift=-3pt]MDRL.west) -- ([yshift=-3pt]MV^bullet.east) node[midway, below=3pt] {$\mathbf{C}$};

\draw[->] (MV) -- (MV^bullet) node[midway, left=2pt] {$\mathbf{F}$};
  \draw[->] (iIRL_0) -- (MDRL) node[midway, right=2pt] {$\mathbf{F}$};
\end{tikzpicture} 
\begin{center}  {\bf{Figure 4.}} Diagram of the functors $\mathbf{K}$ and $\mathbf{C}$ relate the categories $\mathbb{FARDL}$ and $\mathbb{FAD\text{iff}RDL'}$ .
\end{center}
\end{center}

\begin{corollary}\label{cor:representation}
Every Frobenius-adjoint residuated distributive lattice is isomorphic
to the positive cone of a centered Frobenius-adjoint c-differential residuated
distributive lattice satisfying $\mathbf{CK}$.  Conversely, every object of
$\CFADp$ is isomorphic to the Kalman algebra of its positive cone.
\end{corollary}

\begin{proof}
Let $(\mathbf{L},f,g)\in\FARDL$.  By Theorem~\ref{theorem4.6},
$(\Kfun(\mathbf{L}),f_K)\in\CFADp$, and the component
\[
 \rho:(\mathbf{L},f,g)\longrightarrow\Cfun(\Kfun(\mathbf{L}),f_K)
\]
is an isomorphism.  This proves the first assertion.

Conversely, let $\boldsymbol{(\mathfrak{L}},f_c)\in\CFADp$.  The component
\[
 \sigma:
 \boldsymbol{(\mathfrak{L}},f_c)
 \longrightarrow
 \Kfun\Cfun\boldsymbol{(\mathfrak{L}},f_c)
\]
is an isomorphism by Theorem~\ref{theorem5.15}. 
\end{proof}

Because finite-model questions are important for the logic introduced earlier,
we also record how finiteness behaves under the equivalence. 
\begin{corollary}\label{cor:finite}
The equivalence of Theorem~\ref{theorem5.15} restricts to the
finite full subcategories.

 $(1)$~ $(\mathbf{L},f,g)$ is finite if and only if $(\Kfun(\mathbf{L}),f_K)$ is finite;
 
 $(2)$~an algebra $(\boldsymbol{\mathfrak{L}},f_c)\in\CFADp$ is finite if and only if
       its positive cone $\Cfun(\boldsymbol{\mathfrak{L}},f_c)$ is finite.
\end{corollary}

\begin{proof}
(1)~If $\mathbf{L}$ is finite, then $\Kfun(\mathbf{L})\subseteq \mathbf{L}\times \mathbf{L}$ is finite. 

 Conversely,
if $\Kfun(\mathbf{L})$ is finite, then its positive cone is finite, and
\[
 \mathbf{L}\cong\Cfun\Kfun(\mathbf{L})
\]
by Proposition~\ref{proposition5.12.}, hence $\mathbf{L}$ is finite.

(2)~Let $(\boldsymbol{\mathfrak{L}},f_c)\in\CFADp$.  If $\boldsymbol{\mathfrak{L}}$ is
finite, then its subset $\Cfun(\boldsymbol{\mathfrak{L}},f_c)$ is finite.  

Conversely,
if the positive cone is finite, then
\[
 \Kfun\Cfun(\boldsymbol{\mathfrak{L}},f_c)
 \subseteq
 \Cfun(\boldsymbol{\mathfrak{L}},f_c)^2
\]
is finite.  Since $(\boldsymbol{\mathfrak{L}},f_c)$ is isomorphic to the Kalman algebra
by Proposition~\ref{proposition5.13}, it is also finite.
\end{proof}

\section{Logical consequence through positive cones}\label{sec:positive-cones}

\renewcommand{\CFAD}{\mathbb{FA}\text{-}c\mathbb{DRDL}}
\renewcommand{\CFADp}{\mathbb{FA}\text{-}c\mathbb{DRDL}^{\prime}}

The categorical equivalence established in Section~\ref{section4.} has a
precise logical counterpart.  Its source category is \(\FARDL\), whose
objects are Frobenius-adjoint residuated distributive lattices
\((\mathbf L,f,g)\), and its target category is \(\CFADp\), whose objects are
Frobenius-adjoint \(c\)-differential residuated distributive lattices
\((\boldsymbol{\mathfrak L},f_c)\) satisfying \(\mathbf{CK}\).  The
residuated lattice is assumed to be distributive, and the logic
corresponding to the equivalence is obtained from \(\FLFGC\) by adding
lattice distributivity.  

\subsection{The distributive extension of the logic \(\mathbf{FL}^{\mathrm{FGC}}_{ew}\)}
\label{subsec:distributive-calculus}

In this subsection, we introduce the logic \(\mathbf{FL}^{\mathrm{FGC},d}_{ew}\), which is the distributive extension of \(\mathbf{FL}^{\mathrm{FGC}}_{ew}\). Also, we show that \(\mathbf{FL}^{\mathrm{FGC},d}_{ew}\) is an algebraizable logic in the sense of Blok and Pigozzi \cite{BlokPigozzi}.
\medskip

Let \(\mathbf{FL}^{\mathrm{FGC},d}_{ew}\) be the extension of \(\FLFGC\)
obtained by adding the two distributivity schemes
\begin{align}
 \varphi\sqcap(\psi\sqcup\chi)
 &\Leftrightarrow
 (\varphi\sqcap\psi)\sqcup(\varphi\sqcap\chi),
 \tag{D1}\label{eq:dist-one}\\
 \varphi\sqcup(\psi\sqcap\chi)
 &\Leftrightarrow
 (\varphi\sqcup\psi)\sqcap(\varphi\sqcup\chi).
 \tag{D2}\label{eq:dist-two}
\end{align}
The axioms for \(\triangledown\) and \(\blacktriangle\) are exactly
those of Definition~\ref{def:logic-FGC}.  Accordingly,
\(\triangledown\) and \(\blacktriangle\) remain primitive connectives,
interpreted by the right adjoint operator \(f\) and the left adjoint operator \(g\),
respectively.

\begin{proposition}
\label{prop:alg-distributive}
The logic \(\mathbf{FL}^{\mathrm{FGC},d}_{ew}\) is algebraizable with defining
equation
\[
 p\approx1,
\]
and equivalence formula
\[
 p\Leftrightarrow q.
\]
Its equivalent algebraic semantics is \(\FARDL\).  Hence, for every
\(\Gamma\cup\{\varphi\}\subseteq
\mathrm{Fm}_{\mathcal L_{\mathrm{FGC}}}\),
\[
 \Gamma\vdash_{\mathbf{FL}^{\mathrm{FGC},d}_{ew}}\varphi
 \Longleftrightarrow
 \Gamma\models_{\FARDL}\varphi.
\]
\end{proposition}

\begin{proof}
Section~\ref{section3} proves that \(\FLFGC\) is algebraizable with defining
equation \(p\approx1\), equivalence formula \(p\Leftrightarrow q\), and
equivalent algebraic semantics \(\FARL\).  Under the same algebraizing
translations, \eqref{eq:dist-one} and \eqref{eq:dist-two} translate exactly
into the two distributivity identities of the lattice reduct.  Therefore the
equivalent algebraic semantics of the extension is the distributive
subvariety \(\FARDL\).  Since the added schemes are equational, the original
algebraizing pair is unchanged.

For completeness, let \(T\) be a theory of
\(\mathbf{FL}^{\mathrm{FGC},d}_{ew}\) and define
\[
 \varphi\equiv_T\psi
 \Longleftrightarrow
 T\vdash_{\mathbf{FL}^{\mathrm{FGC},d}_{ew}}
   \varphi\Leftrightarrow\psi.
\]
By Proposition~\ref{prop:FGC-monotonicity}, \(\equiv_T\) is a congruence of
the full formula algebra.  The quotient
\(\mathrm{Fm}_{\mathcal L_{\mathrm{FGC}}}/\!\equiv_T\) satisfies the
Frobenius-adjoint equations and the two distributivity identities; hence it
belongs to \(\FARDL\).  The canonical valuation into this quotient gives the
converse semantic implication.
\end{proof}

\begin{remark}
\label{rem:ambient-involution}
Every \((\boldsymbol{\mathfrak L},f_c)\in\CFADp\) has an involution
\({\sim}\), which does not make
\(\mathbf{FL}^{\mathrm{FGC},d}_{ew}\) involutive.  On the positive cone,
logical negation is the residual negation relative to the center,
\[
 \neg_c x:=x\leadsto c,
\]
whereas \({\sim}x\) is computed in the whole centered algebra and may lie
below \(c\).  In general,
\[
 \neg_c x\neq{\sim}x
 \qquad\text{and}\qquad
 \neg_c\neg_c x=x
\]
need not hold.  The involution \({\sim}\) is used only in the term definition
\(g_c(x)={\sim}f_c({\sim}x)\).
\end{remark}

Here we present a positive-cone semantics of the logic \(\mathbf{FL}^{\mathrm{FGC},d}_{ew}\).

\begin{definition}
\label{def:positive-cone-valuation}
Let
\[
 \mathcal{\mathscr L}_c=(\boldsymbol{\mathfrak L},f_c)
 = (\mathcal{L},\cap,\cup,\otimes,\leadsto,{\sim},
    f_c,0,c,1)
 \in\CFADp
\]
and put
\[
 C=\mathcal C(\boldsymbol{\mathfrak L})
   =\{x\in \mathcal{L}\mid x\geq c\}.
\]
The positive-cone algebra associated with \(\mathcal{\mathscr L}_c\) is
\[
 \Cfun(\mathcal{\mathscr L}_c)
 =\bigl(C,\cap,\cup,\otimes_c,\leadsto,c,1,
        f_c|_C,g_c|_C\bigr),
\]
where
\[
 x\otimes_c y=(x\otimes y)\cup c,
 \qquad
 g_c(x)={\sim}f_c({\sim}x).
\]
The logical symbols
\[
 \bot,\ \top,\ \sqcap,\ \sqcup,\ \mathbin{\&},\ \Rightarrow,
 \ \triangledown,\ \blacktriangle
\]
are interpreted, respectively, by
\[
 c,\ 1,\ \cap,\ \cup,\ \otimes_c,\ \leadsto,
 \ f_c|_C,\ g_c|_C.
\]
A positive-cone valuation is a map
\[
 v:\mathsf{Var}\longrightarrow C
\]
extended recursively to all formulas.  We write
\[
 (\mathcal{\mathscr L}_c,v)\models_C\varphi
 \Longleftrightarrow
 v(\varphi)=1.
\]
For \(\Gamma\cup\{\varphi\}\subseteq
\mathrm{Fm}_{\mathcal L_{\mathrm{FGC}}}\), we write
\[
 \Gamma\models_{\Cfun(\CFADp)}\varphi
\]
when every positive-cone valuation in every object of \(\CFADp\) that assigns
value \(1\) to all formulas in \(\Gamma\) also assigns value \(1\) to
\(\varphi\).
\end{definition}

\begin{proposition}
\label{prop:positive-cone-well-defined}
Let \(\mathcal{\mathscr L}_c=(\boldsymbol{\mathfrak L},f_c)\in\CFADp\) and
\(C=\mathcal C(\boldsymbol{\mathfrak L})\).  Then:

(1)\ \(c,1\in C\),

(2)\ \(C\) is closed under \(\cap\), \(\cup\), \(\otimes_c\), and
      \(\leadsto\),
      
(3)\ \(f_c(C)\subseteq C\) and \(g_c(C)\subseteq C\),

(4)\ \(g_c\dashv f_c\) on \(C\),

(5)\ \(\Cfun(\mathcal{\mathscr L}_c)\in\FARDL\),

(6)\ every formula evaluated under a positive-cone valuation has its value
      in \(C\).
\end{proposition}

\begin{proof}
The first assertion is immediate.  The interval \([c,1]\) is closed under
\(\cap\) and \(\cup\), and
\[
 x\otimes_c y=(x\otimes y)\cup c\geq c.
\]
The positive-cone construction recalled in Section~\ref{section:preliminaries-RL} shows that
\(x\leadsto y\in C\) whenever \(x,y\in C\), and that \(\leadsto\) is the
residual of \(\otimes_c\) on \(C\).

If \(x\in C\), then \(x=x\cup c\).  By axiom \eqref{fc6},
\[
 f_c(x)=f_c(x\cup c)=f_c(x)\cup c,
\]
so \(f_c(x)\geq c\).  Furthermore, \(x\geq c\) implies
\({\sim}x\leq c\).  By the monotonicity of \(f_c\) and \(f_c(c)=c\),
\[
 f_c({\sim}x)\leq f_c(c)=c.
\]
Applying the order-reversing involution gives
\[
 g_c(x)={\sim}f_c({\sim}x)\geq c.
\]
Thus both unary operations preserve the cone.

The results of Section~\ref{section4.} show that the restrictions of \(f_c\)
and \(g_c\) to \(C\) satisfy \eqref{FARL1}--\eqref{FARL5}; in particular,
\(g_c\dashv f_c\).  The lattice reduct is distributive, so
\(\Cfun(\mathcal{\mathscr L}_c)\in\FARDL\).  The last assertion now follows
by induction on formula complexity.
\end{proof}

\begin{lemma}
\label{lem:formula-transport}
Let
\[
 h:(\mathbf L_1,f_1,g_1)\longrightarrow(\mathbf L_2,f_2,g_2)
\]
be a homomorphism in \(\FARDL\), and let \(v\) be a valuation in
\((\mathbf L,f,g)\).  Define \(h(v)\) on variables by
\[
 h(v)(p)=h(v(p)).
\]
Then, for every formula \(\varphi\),
\[
 h(v)(\varphi)=h(v(\varphi)).
\]
If \(h\) is an isomorphism, then
\[
 v(\varphi)=1
 \Longleftrightarrow
 h(v)(\varphi)=1.
\]
\end{lemma}

\begin{proof}
The proof is a structural induction.  The variable and constant cases are
immediate.  The inductive steps for \(\sqcap\), \(\sqcup\), \(\mathbin{\&}\),
and \(\Rightarrow\) follow because \(h\) preserves the residuated-lattice
operations.  For the unary connectives,
\begin{align*}
 h(v)(\triangledown\varphi)
 &=f_2(h(v)(\varphi))
  =f_2(h(v(\varphi)))
  =h(f_1(v(\varphi))),\\
 h(v)(\blacktriangle\varphi)
 &=g_2(h(v)(\varphi))
  =g_2(h(v(\varphi)))
  =h(g_1(v(\varphi))).
\end{align*}
This completes the induction.  The final assertion follows from
\(h(1)=1\) and injectivity.
\end{proof}

We then construct two natural isomorphisms, serving as the basis for the proof of completeness.

\begin{proposition}
\label{prop:rho-formula-transport}
Let \((\mathbf L,f,g)\in\FARDL\), let \(v\) be a valuation in \(L\), and
define a positive-cone valuation in \((\mathbf K(\mathbf{L}),f_K)\) by
\[
 v^K(p)=\rho_L(v(p))=(v(p),0).
\]
Then, for every formula \(\varphi\),
\begin{equation}
 v^K(\varphi)=\rho_L(v(\varphi))=(v(\varphi),0).
 \label{eq:rho-formula-value}
\end{equation}
Consequently,
\[
 v(\varphi)=1
 \quad\Longleftrightarrow\quad
 v^K(\varphi)=(1,0).
\]
\end{proposition}

\begin{proof}
The map \(\rho_L\) is the natural isomorphism from \((\mathbf L,f,g)\) onto
the positive cone of \((\mathbf K(\mathbf{L}),f_K)\).  Hence the assertion follows
from Lemma~\ref{lem:formula-transport}.  Directly, one has
\[
 (x,0)\otimes_c(y,0)=(x\odot y,0),
 \qquad
 (x,0)\leadsto(y,0)=(x\to y,0),
\]
\[
 f_K(x,0)=(f(x),g(0))=(f(x),0),
\]
and, since \({\sim}(x,y)=(y,x)\),
\[
 {\sim}f_K({\sim}(x,0))
 ={\sim}f_K(0,x)
 ={\sim}(0,g(x))
 =(g(x),0).
\]
Thus every connective preserves pairs of the form \((u,0)\), proving
\eqref{eq:rho-formula-value} by induction.
\end{proof}

\begin{proposition}
\label{prop:sigma-formula-transport}
Let \(\mathcal{\mathscr L}_c=(\boldsymbol{\mathfrak L},f_c)\in\CFADp\), and
let
\[
 \sigma_{\boldsymbol{\mathfrak L}}:
 \boldsymbol{\mathfrak L}\longrightarrow
 \mathbf K\bigl(\mathcal C(\boldsymbol{\mathfrak L})\bigr),
 \qquad
 \sigma_{\boldsymbol{\mathfrak L}}(x)
 =(x\cup c,{\sim}x\cup c),
\]
be the natural isomorphism of Section~\ref{section4.}.  If
\(x\in\mathcal C(\boldsymbol{\mathfrak L})\), then
\[
 \sigma_{\boldsymbol{\mathfrak L}}(x)=(x,c).
\]
Consequently, if \(v\) is a positive-cone valuation and
\[
 v^{\sigma}(p)=\sigma_{\boldsymbol{\mathfrak L}}(v(p)),
\]
then
\begin{equation}
 v^{\sigma}(\varphi)
 =\sigma_{\boldsymbol{\mathfrak L}}(v(\varphi))
 =(v(\varphi),c)
 \label{eq:sigma-formula-value}
\end{equation}
for every formula \(\varphi\).
\end{proposition}

\begin{proof}
If \(x\geq c\), then \(x\cup c=x\).  Since \({\sim}\) reverses the order
and fixes \(c\), one has \({\sim}x\leq c\), so
\({\sim}x\cup c=c\).  Therefore
\(\sigma_{\boldsymbol{\mathfrak L}}(x)=(x,c)\).  The restriction of
\(\sigma_{\boldsymbol{\mathfrak L}}\) to the positive cone preserves all
operations, including \(f_c\) and \(g_c\).  Lemma~\ref{lem:formula-transport}
therefore gives \eqref{eq:sigma-formula-value}.
\end{proof}

\begin{corollary}
\label{cor:equivalence-invariance}
It is worth noticing that validity, local consequence, equation satisfaction, and quasi-equation
satisfaction are invariant under the natural isomorphisms \(\rho\) and
\(\sigma\) of
\[
 \FARDL\simeq\CFADp.
\]
\end{corollary}

Now we prove the completeness and Compactness of the logic \(\mathbf{FL}^{\mathrm{FGC},d}_{ew}\).

\begin{theorem}
\label{thm:positive-cone-completeness}
For every \(\Gamma\cup\{\varphi\}\subseteq
\mathrm{Fm}_{\mathcal L_{\mathrm{FGC}}}\),
\[
 \Gamma\vdash_{\mathbf{FL}^{\mathrm{FGC},d}_{ew}}\varphi
 \Longleftrightarrow\
 \Gamma\models_{\FARDL}\varphi
 \Longleftrightarrow
 \Gamma\models_{\Cfun(\CFADp)}\varphi.
\]
\end{theorem}

\begin{proof}
The first equivalence is Proposition~\ref{prop:alg-distributive}.

Assume \(\Gamma\models_{\FARDL}\varphi\).  Let
\(\mathcal{\mathscr L}_c\in\CFADp\), and let \(v\) be a positive-cone
valuation such that \(v(\gamma)=1\) for all \(\gamma\in\Gamma\).
Proposition~\ref{prop:positive-cone-well-defined} gives
\(\Cfun(\mathcal{\mathscr L}_c)\in\FARDL\).  Hence \(v(\varphi)=1\), and
therefore
\(\Gamma\models_{\Cfun(\CFADp)}\varphi\).

Conversely, assume
\(\Gamma\models_{\Cfun(\CFADp)}\varphi\).  Let
\((\mathbf L,f,g)\in\FARDL\), and let \(v\) be a valuation such that
\(v(\gamma)=1\) for every \(\gamma\in\Gamma\).  By Section~\ref{section4.},
\((\mathbf K(\mathbf{L}),f_K)\in\CFADp\).  Define
\(v^K(p)=(v(p),0)\).  Proposition~\ref{prop:rho-formula-transport} gives
\[
 v^K(\gamma)=(1,0)
 \qquad(\gamma\in\Gamma).
\]
Positive-cone validity yields \(v^K(\varphi)=(1,0)\).  Again by
Proposition~\ref{prop:rho-formula-transport},
\[
 (v(\varphi),0)=(1,0),
\]
so \(v(\varphi)=1\).  Therefore
\(\Gamma\models_{\FARDL}\varphi\).
\end{proof}

\begin{corollary}
\label{cor:positive-global-completeness}
For every formula \(\varphi\),
\[
 \vdash_{\mathbf{FL}^{\mathrm{FGC},d}_{ew}}\varphi
 \quad\Longleftrightarrow\quad
 \models_{\FARDL}\varphi
 \quad\Longleftrightarrow\quad
 \models_{\Cfun(\CFADp)}\varphi.
\]
\end{corollary}

\begin{corollary}
\label{cor:positive-compactness}
For every \(\Gamma\cup\{\varphi\}\subseteq
\mathrm{Fm}_{\mathcal L_{\mathrm{FGC}}}\),
\[
 \Gamma\models_{\Cfun(\CFADp)}\varphi
\]
if and only if there exists a finite \(\Gamma_0\subseteq\Gamma\) such that
\[
 \Gamma_0\models_{\Cfun(\CFADp)}\varphi.
\]
\end{corollary}

\begin{proof}
By Theorem~\ref{thm:positive-cone-completeness}, semantic consequence in the
positive cones coincides with derivability in
\(\mathbf{FL}^{\mathrm{FGC},d}_{ew}\).  Every Hilbert derivation uses only
finitely many premises.  Applying the completeness theorem once more gives
the result.
\end{proof}

\subsection{Equational consequence and conservativity}
\label{subsec:equational-consequence}

Having established completeness through positive-cone semantics, in this subsection, we show that the same correspondence preserves equational and quasi-equational consequence of \(\mathbf{FL}^{\mathrm{FGC},d}_{ew}\).

\begin{theorem}
\label{thm:equational-transport}
Let \(\Sigma\cup\{s\approx t\}\) be a set of equations in the source
signature
\[
 (\wedge,\vee,\odot,\to,0,1,f,g).
\]
Then
\[
 \Sigma\models_{\FARDL}s\approx t
\]
if and only if the corresponding equation is valid in every positive-cone
algebra \(\Cfun(\mathcal{\mathscr L}_c)\), where
\(\mathcal{\mathscr L}_c\in\CFADp\).  The same equivalence holds for
quasi-equations.
\end{theorem}

\begin{proof}
If \(\Sigma\models_{\FARDL}s\approx t\), the conclusion holds in every
positive-cone algebra because each such algebra belongs to \(\FARDL\).

Conversely, let \((\mathbf L,f,g)\in\FARDL\) and let an assignment \(v\)
satisfy every equation in \(\Sigma\).  Transport the assignment to the
positive cone of \(\mathbf K(\mathbf{L})\) by
\[
 v^K(p)=(v(p),0).
\]
The term version of Proposition~\ref{prop:rho-formula-transport} gives
\[
 v^K(r)=(v(r),0)
\]
for every term \(r\).  Hence \(v^K\) satisfies \(\Sigma\), and the assumed
positive-cone consequence gives
\[
 (v(s),0)=v^K(s)=v^K(t)=(v(t),0).
\]
Thus \(v(s)=v(t)\).  The proof for quasi-equations is identical, since the
transport preserves each premise equation and the conclusion equation.
\end{proof}

\begin{theorem}
\label{thm:positive-conservativity}
Suppose that \(\Gamma\cup\{\varphi\}\) contains no occurrences of
\(\triangledown\) or \(\blacktriangle\).  Then
\[
 \Gamma\vdash_{\mathbf{FL}^{\mathrm{FGC},d}_{ew}}\varphi
 \Longleftrightarrow
 \Gamma\vdash_{\mathbf{FL}^{d}_{ew}}\varphi
 \Longleftrightarrow
 \Gamma\models_{\Cfun(\CFADp)}\varphi,
\]
where \(\mathbf{FL}^{d}_{ew}\) is the distributive extension of
\(\FLew\), without an involution axiom.
\end{theorem}

\begin{proof}
The implication from the base calculus to its expansion is immediate.
Conversely, suppose
\[
 \Gamma\nvdash_{\mathbf{FL}^{d}_{ew}}\varphi.
\]
By strong completeness of the distributive base logic, there exist a
residuated distributive lattice \(\mathbf L\) and a valuation \(v\) such that
\[
 v(\gamma)=1\quad(\gamma\in\Gamma),
 \qquad
 v(\varphi)\neq1.
\]
Expand \(\mathbf L\) by \(f=g=\Id_L\).  Then
\((\mathbf L,f,g)\in\FARDL\).  Since the formulas contain no unary
connectives, their values are unchanged; hence
\[
 \Gamma\not\models_{\FARDL}\varphi.
\]
Proposition~\ref{prop:alg-distributive} and
Theorem~\ref{thm:positive-cone-completeness} yield the displayed
equivalences.
\end{proof}

\subsection{Finite countermodels}
\label{subsec:finite-countermodels}

We finally show that the finite model properties formulated in \(\FARDL\) coincide with those formulated through finite positive cones of algebras in \(\CFADp\) based on Kalman equivalence in Subsection \ref{subsection:FMP-FGC} and Subsection\ref{subsection:CAE-FARL}.

\begin{theorem}
\label{thm:finite-countermodel-transfer}
For every \(\Gamma\cup\{\varphi\}\subseteq
\mathrm{Fm}_{\mathcal L_{\mathrm{FGC}}}\), the following statements are equivalent:

(1)\ there exist a finite \((\mathbf L,f,g)\in\FARDL\) and a valuation
      \(v\) such that
      \[
       v(\gamma)=1\quad(\gamma\in\Gamma),
       \qquad
       v(\varphi)\neq1;
      \]
      
(2)\ there exist a finite
      \(\mathcal{\mathscr L}_c=(\boldsymbol{\mathfrak L},f_c)\in\CFADp\)
      and a positive-cone valuation \(w\) such that
      \[
       w(\gamma)=1\quad(\gamma\in\Gamma),
       \qquad
       w(\varphi)\neq1.
      \]
Consequently, the finite model property and the strong finite model property
for \(\mathbf{FL}^{\mathrm{FGC},d}_{ew}\) are equivalent whether they are
formulated by finite members of \(\FARDL\) or by finite positive cones of
members of \(\CFADp\).
\end{theorem}

\begin{proof}
Assume item~(1).  The Kalman expansion \((\mathbf K(\mathbf{L}),f_K)\) belongs to
\(\CFADp\) and is finite.  Define
\[
 w(p)=(v(p),0).
\]
By Proposition~\ref{prop:rho-formula-transport},
\[
 w(\alpha)=(v(\alpha),0)
\]
for every formula \(\alpha\).  Thus \(w\) is the required finite
positive-cone countermodel.

Conversely, assume item~(2).  By
Proposition~\ref{prop:positive-cone-well-defined},
\(\Cfun(\mathcal{\mathscr L}_c)\in\FARDL\).  Its universe is the finite set
\(\mathcal C(\boldsymbol{\mathfrak L})\), and the same map \(w\), regarded
as a valuation in this source algebra, witnesses item~(1).
\end{proof}

Let \(\FARDL_{\mathrm{fin}}\) denote the class of finite members of
\(\FARDL\).

\begin{theorem}
\label{thm:finite-model-characterization}
The following conditions are equivalent:

(1)\ \(\mathbf{FL}^{\mathrm{FGC},d}_{ew}\) has the finite model property;

(2)\ \(\FARDL=\operatorname{V}(\FARDL_{\mathrm{fin}})\);

(3)\ every finitely generated free \(\FARDL\)-algebra is residually finite;

(4)\ every non-theorem of \(\mathbf{FL}^{\mathrm{FGC},d}_{ew}\) has a
      finite positive-cone countermodel in \(\CFADp\).
\end{theorem}

\begin{proof}
The equivalence of~(1)--(3) is the distributive specialization of the
finite-model characterization proved in Section~\ref{section3}.  The
equivalence between~(1) and~(4) is
Theorem~\ref{thm:finite-countermodel-transfer} with \(\Gamma=\varnothing\).
\end{proof}

\begin{proposition}
\label{prop:finite-embeddability-criterion}
If \(\FARDL\) has the finite embeddability property, then
\(\mathbf{FL}^{\mathrm{FGC},d}_{ew}\) has the strong finite model property.
Equivalently, every failed finite consequence has a finite positive-cone
countermodel in a member of \(\CFADp\).
\end{proposition}

\begin{proof}
Let \(\Gamma\) be finite and suppose
\[
 \Gamma\nvdash_{\mathbf{FL}^{\mathrm{FGC},d}_{ew}}\varphi.
\]
Choose \((\mathbf L,f,g)\in\FARDL\) and a valuation \(v\) with all premises
of value \(1\) and conclusion of value different from \(1\).  Let \(\Sigma\)
be the finite set of subformulas of formulas in
\(\Gamma\cup\{\varphi\}\), and put
\[
 P=\{v(\psi):\psi\in\Sigma\}\cup\{0,1\}.
\]
Restrict each basic operation to tuples from \(P\) whose value remains in
\(P\).  This is a finite partial subalgebra.  By the finite embeddability
property, it embeds into a finite member of \(\FARDL\).  A structural
induction preserves all values of relevant subformulas and yields a finite
source countermodel.  Theorem~\ref{thm:finite-countermodel-transfer} converts
it into a finite positive-cone countermodel.
\end{proof}

\section{Concluding remarks}

In the present paper, we introduced the Frobenius--Galois logic $\mathbf{FL}^{\mathbf{FGC}}_{ew}$ and identified $\FARL$ as its equivalent algebraic semantics. In the distributive setting, the Kalman and positive-cone functors yield a categorical equivalence between $\FARDL$ and $\CFADp$, and this equivalence induces a positive-cone semantics for $\mathbf{FL}^{\mathbf{FGC},d}_{ew}$ preserving logical consequence, equational consequence, and finite countermodels.  Here we record two problems that seem particularly relevant to the algebraic, categorical, and logical structure of Frobenius--Galois expansions.

\begin{openproblem}\label{op:nondistributive-Kalman}
Can the Kalman equivalence established in Theorem~\ref{theorem5.15}
be extended from $\FARDL$ to the non-distributive variety $\FARL$?
More precisely, does there exist a category $\mathcal K$ of 
 Frobenius-adjoint centered involutive  residuated lattices and functors
\[
\mathbf K:\FARL\longrightarrow\mathcal K,
\qquad
\mathbf C:\mathcal K\longrightarrow\FARL
\]
such that
\[
\FARL\simeq\mathcal K,
\]
and whose restriction to the distributive subcategories is naturally
isomorphic to the equivalence of Theorem~\ref{theorem5.15}? If an
equivalence does not exist on $\FARL$, characterize the
subvarieties of $\FARL$ for which a Kalman-type equivalence exists and
determine whether there is a largest such subvariety.
\end{openproblem}

\begin{openproblem}\label{op:logical-Kalman-translation}
Does there exist an algebraizable logic $\mathcal L_c$ whose equivalent
algebraic semantics is $\CFADp$ and a finite contextual translation
between $\mathbf{FL}^{\mathbf{FGC},d}_{ew}$ and $\mathcal L_c$ that
preserves and reflects derivability in both directions? More precisely,
can one construct translations $\tau$ and $\rho$ such that the induced
algebraic correspondence agrees, up to natural isomorphism, with the
Kalman and positive-cone functors $\mathbf K$ and $\mathbf C$ of
Theorem~\ref{theorem5.15}?
\end{openproblem}
\medskip
\noindent{\bf CRediT authorship contribution statement}
\medskip

\noindent \textbf{Juntao Wang:} Writing--review \& editing, Writing--original draft, \textbf{Jieqiong Shi:}Writing--review \& editing, \textbf{Mei Wang:}Writing--review \& editing.

\medskip
\noindent{\bf Acknowledgments} 
\medskip

This work is supported by the National Natural Science Foundation of China (12501647), the Humanities and Social Sciences Youth Foundation of Ministry of Education of China ``A propositional calculus formal study of monadic substructural logics''(24XJC72040001), and the Natural Science Basic Research Plan in Shaanxi Province of China (2025JC-YBMS-034, 2025JC-YBQN-092), and the Scientific Research Program Funded by Education Department of Shaanxi Provincial Government (Youth Innovation Team of Shaanxi Universities) (23JP132).

\medskip
\noindent{\bf Conflict of interest} 
\medskip

The authors declare that they have no Conflict of interest.

\medskip
\noindent{\bf Data availability} 
\medskip

 No data was used for the research described in the article.

\end{document}